\documentclass[a4paper, 12pt]{article}

\usepackage{amsfonts, amsmath, amsthm}
\usepackage{amssymb}
\usepackage{latexsym}
\usepackage[dvips]{graphicx}
\usepackage{color}

\theoremstyle{plain}
\newtheorem{thm}{Theorem}[section]
\newtheorem{prp}{Proposition}[section]
\newtheorem{lem}{Lemma}[section]

\theoremstyle{definition}

\theoremstyle{remark}
\newtheorem{rmk}{Remark}[section]
\newtheorem{ex}{Example}[section]

\numberwithin{equation}{section}
\newcommand{\Z}{\mathbb{Z}}

\newcommand{\R}{\mathbb{R}}
\newcommand{\C}{\mathbb{C}}

\newcommand{\pa}{\partial}
\newcommand{\eps}{\varepsilon}
\newcommand{\jb}[1]{\langle #1 \rangle}
\newcommand{\Jb}[1]{\bigl\langle #1 \bigr\rangle}
\newcommand{\JB}[1]{\Bigl\langle #1 \Bigr\rangle}

\newcommand{\sh}[1]{{#1}^{\sharp}}
\newcommand{\tm}{\tilde{m}}

\DeclareMathOperator{\realpart}{\rm Re}
\DeclareMathOperator{\imagpart}{\rm Im}

\newcommand{\dis}{\displaystyle}

\begin{document}
%%%%%%%%%%%%%%%%%%%%%%%%%%%%%%%%%%%%%%%%%%%%%%%%%%%%%%%%%%%%%%%%%%%%%%%%%%%%%%%
\title{
 On Schr\"odinger systems with cubic 
dissipative nonlinearities of derivative type\\
 }

%-----------------------
\author{
           Chunhua Li\thanks{
          %\underline{Chunhua Li} \thanks{
              Department of Mathematics, College of Science, 
              Yanbian University. 
              977 Gongyuan Road, Yanji, Jilin Province, 133002, China. 
              (E-mail: {\tt sxlch@ybu.edu.cn})
             }
           \and  
          Hideaki Sunagawa \thanks{
          %\underline{Hideaki Sunagawa} \thanks{
              Department of Mathematics, Graduate School of Science, 
              Osaka University. 
              1-1 Machikaneyama-cho, Toyonaka, Osaka 560-0043, Japan. 
              (E-mail: {\tt sunagawa@math.sci.osaka-u.ac.jp})
             }
}%-----------------------

\date{\today }   
\maketitle

%-----------------------
\noindent{\bf Abstract:}\ 
Consider the initial value problem for systems of cubic derivative 
nonlinear Schr\"odinger equations in one space dimension with the masses 
satisfying a suitable resonance relation. We give 
structural conditions on the nonlinearity under which the small data 
solution gains an additional logarithmic decay as $t \to +\infty$ 
compared with the corresponding free evolution.  \\

%-----------------------

%-----------------------

%-------------------------------------------------------------------
\section{Introduction}  \label{sec_intro}
%--------------------------------------------------------------------
Consider the initial value problem for the system of nonlinear Schr\"odinger equations of the following type:
\begin{align}
\left\{\begin{array}{cl}
\mathcal{L}_{m_j} u_j=F_j(u,\pa_x u), & t>0,\ x\in \R,\ j=1,\ldots, N,\\
u_j(0,x)=\varphi_j(x), & x \in \R,\ j=1,\ldots, N,
\end{array}\right.
\label{nls_system}
\end{align}
where $\mathcal{L}_{m_{j}}=i\pa_t +\frac{1}{2m_{j}}\pa_x^2$, 
$i=\sqrt{-1}$, $m_j \in \R\backslash \{0\}$, and 
$u=(u_j(t,x))_{1\le j \le N}$ is a $\C^N$-valued unknown function. 
The nonlinear term $F=(F_j)_{1\le j\le N}$ is always assumed to be 
a cubic homogeneous polynomial in 
$(u,\pa_x u, \overline{u}, \overline{\pa_x u})$. 
Our main interest is how the combinations of $(m_j)_{1\le j \le N}$ and 
the structures of $(F_j)_{1\le j \le N}$ affect large-time behavior of 
the solution $u$ to \eqref{nls_system}. 
Before going into details, let us first recall some known results briefly 
and clarify our motivation.

One of the most typical nonlinear Schr\"odinger equations appearing in 
various physical settings is 
\begin{align}
 i\pa_t u + \frac{1}{2}\pa_x^2 u =\lambda |u|^2 u, \qquad
 t>0,\ x\in \R
 \label{nls_single}
\end{align}
with $\lambda \in \R$. What is interesting in \eqref{nls_single} is 
that the large-time behavior of the solution is actually affected by the 
nonlinearity even if the initial data is sufficiently small, smooth and 
decaying fast as $|x|\to \infty$. To be more precise, it is shown in \cite{HN} 
that the solution to \eqref{nls_single} with small initial data behaves like 
\[
 u(t,x)=\frac{1}{\sqrt{it}} \alpha(x/t) 
 e^{i\{\frac{x^2}{2t}  - \lambda |\alpha(x/t)|^2 \log t \}}
 +o(t^{-1/2})
\quad \mbox{as} \ \ t\to \infty
\]
with a suitable $\C$-valued function $\alpha(y)$. 
An important consequence of this asymptotic expression is that the solution 
decays like $O(t^{-1/2})$ in $L^{\infty}(\R_{x})$, 
while it does not behave like the free solution unless $\lambda=0$. 
In other words, the additional logarithmic factor in the phase 
reflects the long-range character of the cubic nonlinear Schr\"odinger 
equations in one space dimension.  If $\lambda \in \C$, another kind of 
long-range effect can be observed. Indeed, it is verified in \cite{Shim} 
that the small data solution to \eqref{nls_single} decays like 
$O(t^{-1/2} (\log t)^{-1/2})$ in $L^{\infty}(\R_{x})$ as $t \to \infty$ 
if $\imagpart \lambda <0$ (see also \cite{Su2}). 
This gain of additional logarithmic time decay should be interpreted 
as another kind of long-range effect. 
Among  several extensions of this result 
(see e.g., \cite{HNS}, \cite{Kim}, \cite{KitaShim1}, \cite{KitaShim2}, 
\cite{LH} etc. and the references cited therein), 
let us focus on the following two cases: 
(i) the case where the nonlinearity depends also on $\pa_x u$, 
and (ii) the case of systems. 

\begin{itemize}

\item[(i)] 
Let us consider the single nonlinear Schr\"odinger equation
\begin{align}
 i\pa_t u + \frac{1}{2}\pa_x^2 u = G(u,\pa_x u), \qquad 
 t>0,\ x\in \R,
 \label{dnls_single}
\end{align}
where $G$ is a cubic homogeneous polynomial in 
$(u,\pa_x u, \overline{u}, \overline{\pa_x u})$ with complex 
coefficients, and satisfies the gauge invariance 
\begin{align}
 G(e^{i\theta} v,e^{i\theta} w) =e^{i\theta} G(v,w),
 \qquad 
 \theta\in \R,\ (v,w) \in \C \times \C.
 \label{gi}
\end{align}
According to \cite{HNS}, the solution to \eqref{dnls_single} decays like 
$O(t^{-1/2} (\log t)^{-1/2})$ in $L^{\infty}(\R_{x})$ as $t \to \infty$ if 
\begin{align}
 \sup_{\xi \in \R} \imagpart G(1,i\xi) <0.
\label{dissip_single}
\end{align}
However, the approach of \cite{HNS} does not work well in the case of systems, 
because this additional logarithmic decay result is a consequence of the 
explicit asymptotic profile of the solution $u(t,x)$, 
which becomes no longer simple in the coupled case.

\item[(ii)] 
For nonlinear Schr\"odinger systems, an additional logarithmic decay result 
is first obtained by \cite{KLS}. 
Strictly saying, two-dimensional quadratic nonlinear Schr\"odinger systems 
are treated in \cite{KLS}, but we can adopt the method of \cite{KLS} 
directly to one-dimensional cubic nonlinear Schr\"odinger systems, 
as pointed in \cite{Kim}. When we restrict ourselves to a two-component model
\begin{align}
 \left\{\begin{array}{l}
 \mathcal{L}_{m_1} u_1 =\lambda_1 |u_1|^2 u_1 + \nu_1 \overline{u_1}^2 u_2,\\
 \mathcal{L}_{m_2} u_2 =\lambda_2 |u_2|^2 u_2 + \nu_2 u_1^3,
 \end{array}\right.
\qquad 
 t>0,\ x\in \R \label{nls_two}
\end{align}
with $\lambda_1$, $\lambda_2$, $\nu_1$, $\nu_2 \in \C$
and $m_{1},m_{2} \in \R \backslash \{0\}$, 
then the result of \cite{KLS} can be read as follows: 
the solution to \eqref{nls_two} decays like $O(t^{-1/2} (\log t)^{-1/2})$ in 
$L^{\infty}(\R_{x})$ as $t \to \infty$ if 
\begin{align}
 & m_2=3m_1, \label{mass_resonance}\\
 & \imagpart \lambda_j<0,\qquad j=1,2, \label{dissip}
\end{align}
and 
\begin{align}
 \kappa_1 \nu_1=\kappa_2 \overline{\nu_2}
 \quad \mbox{with some $\kappa_1$, $\kappa_2>0$}
 \label{interaction}
\end{align}
(see Example 2.1 in \cite{Kim} for the detail).
The advantage of the method of \cite{KLS} is that it does not 
rely on the explicit asymptotic profile at all. However, it is 
not straightforward to apply this approach in the derivative 
nonlinear case, because we need suitable pointwise a priori estimates 
not only for the solution itself but also for its derivatives 
without breaking good structure in order to apply the 
method of \cite{KLS}.
\end{itemize}

The purpose of this paper is to unify  (i) and (ii). 
More precisely, we will introduce structural conditions on 
$(F_j)_{1\le j \le N}$ and $(m_j)_{1\le j \le N}$ 
under which the small data solution to the derivative nonlinear 
Schr\"odinger system \eqref{nls_system} gains an additional logarithmic 
decay as $t \to +\infty$ compared with the corresponding free evolution.  

%-------------------------------------------------------------------
\section{Main Results}  \label{sec_result}
%--------------------------------------------------------------------
In the subsequent sections, we will use the following notations: 
We set $I_N=\{1,\ldots, N\}$ and 
$\sh{I}_N =\{1,\ldots, N, N+1,\ldots, 2N \}$. 
For $z =(z_j)_{j\in I_N}\in \C^N$, we write 
\[
 \sh{z} =(\sh{z}_k)_{k \in \sh{I}_N}
 :=(z_1,\ldots,z_N, \overline{z_1},\ldots, \overline{z_N})
\in \C^{2N}.
\]
Then general cubic nonlinear term $F=(F_j)_{j \in I_N}$ can be written as 
\begin{align*}
 F_j(u,\pa_x u)
 =\sum_{l_1, l_2,l_3=0}^{1}\sum_{k_1, k_2, k_3 \in \sh{I}_N}
 C_{j, k_1, k_2, k_3}^{l_1,l_2,l_3} (\pa_x^{l_1} \sh{u}_{k_1}) 
 (\pa_x^{l_2} \sh{u}_{k_2}) (\pa_x^{l_3} \sh{u}_{k_3})
\end{align*}
with suitable $C_{j, k_1, k_2, k_3}^{l_1,l_2,l_3} \in \C$. 
With this expression of $F$, we define 
$p=(p_j(\xi;Y))_{j \in I_N}:\R \times \C^N \to \C^N$ by  
\begin{align*}
 p_j(\xi; Y)
 :=
 \sum_{l_1, l_2,l_3=0}^{1}\sum_{k_1, k_2, k_3 \in \sh{I}_N}
 C_{j, k_1, k_2, k_3}^{l_1,l_2,l_3} (i\tm_{k_1} \xi)^{l_1}
 (i\tm_{k_2} \xi)^{l_2} (i\tm_{k_3} \xi)^{l_3} 
 \sh{Y}_{k_1} \sh{Y}_{k_2} \sh{Y}_{k_3} 
\end{align*}
for $\xi \in \R$ and $Y=(Y_j)_{j \in I_N} \in \C^N$, where 
\[
 \tm_k=\left\{\begin{array}{cl}
 m_k& (k=1,\ldots, N),\\[4mm]
 -m_{(k-N)} & (k = N+1,\ldots, 2N).
 \end{array}\right.
\]
In what follows, we denote by $\jb{\cdot, \cdot}_{\C^N}$ the standard scalar 
product in $\C^N$, i.e., 
$$
 \jb{z,w}_{\C^N}=\sum_{j=1}^{N}  z_j \overline{w_j}
$$
for $z=(z_j)_{j \in I_N}$ and $w=(w_j)_{j \in I_N} \in \C^N$. 

Now let us introduce the following conditions: 
%---------------------------------------------%
\begin{itemize}
%-----------------%
\item[(a)] For all $j \in I_N$ and 
$k_1,k_2, k_3 \in \sh{I}_N$, 
\[
 m_j\ne  \tm_{k_1} + \tm_{k_2} + \tm_{k_3}
 \ \ \mbox{implies}\ \ 
 C_{j,k_1,k_2,k_3}^{l_1,l_2,l_3}=0,\ \ 
 l_1, l_2, l_3 \in \{0,1\}.
\]
%-----------------%
\item[(b$_0$)] 
There exists an $N\times N$ positive Hermitian matrix $A$ such that 
\[
 \imagpart \jb{p(\xi;Y), AY}_{\C^N}\le 0 
\]
for all $(\xi,Y) \in \R\times \C^N$. 
%-----------------%
\item[(b$_1$)] 
There exist an $N\times N$ positive Hermitian matrix $A$ and 
a positive constant $C_*$ such that 
\[
 \imagpart \jb{p(\xi;Y), AY}_{\C^N}\le -C_* |Y|^4 
\]
for all $(\xi,Y) \in \R\times \C^N$. 
%-----------------%
\item[(b$_2$)] 
There exist an $N\times N$ positive Hermitian matrix $A$ and 
a positive constant $C_{**}$ such that 
\[
 \imagpart \jb{p(\xi;Y), AY}_{\C^N}\le -C_{**} \jb{\xi}^2 |Y|^4
\]
for all $(\xi,Y) \in \R\times \C^N$, where $\jb{\xi}=\sqrt{1+\xi^2}$. 
%-----------------%
\item[(b$_3$)] 
$p(\xi;Y)=0$ for all $(\xi,Y) \in \R\times \C^N$. 
\end{itemize}
%-------------------------------------------%

To state the main results, we introduce some function spaces. For 
$s, \sigma \in \Z_{\ge 0}$, 
we denote by $H^s$ the $L^2$-based Sobolev space of order $s$, 
and the weighted Sobolev space $H^{s,\sigma}$ is defined by 
$\{\phi \in L^2\, |\, \jb{x}^{\sigma} \phi \in H^s \}$, equipped with 
the norm $\|\phi\|_{H^{s,\sigma}}=\|\jb{x}^{\sigma} \phi\|_{H^s}$. 
The main results are as follows: 
%-------------------------
\begin{thm} \label{thm_sdge}
Assume the conditions (a) and (b$_0$) are satisfied. Let 
$\varphi=(\varphi_j)_{j \in I_ N} \in H^3\cap H^{2,1}$, 
and assume $\eps:=\|\varphi\|_{H^3}+\|\varphi\|_{H^{2,1}}$ is sufficiently 
small. Then \eqref{nls_system} admits a unique global solution 
$u=(u_j)_{j \in I_N} \in C([0,\infty); H^3\cap H^{2,1})$. 
Moreover we have
\[
 \|u(t)\|_{L^{\infty}} \le \frac{C\eps}{\sqrt{1+t}}, \qquad 
 \|u(t)\|_{L^2} \le C\eps
\]
for $t\ge 0$, where $C$ is a positive constant not depending on $\eps$.\\
\end{thm}
%-------------------------

%-------------------------
\begin{thm} \label{thm_decay1}
Assume the conditions (a) and (b$_1$) are satisfied. Let $u$ 
be the global solution to \eqref{nls_system}, whose existence is guaranteed 
by Theorem \ref{thm_sdge}. Then we have 
\[
 \|u(t)\|_{L^{\infty}} \le \frac{C\eps }{\sqrt{(1+t)\{1+ \eps^2\log (2+t)\}}}
\]
for $t\ge 0$, where $C$ is a positive constant not depending on $\eps$. 
We also have
\[
 \lim_{t\to +\infty} \|u(t)\|_{L^2}=0.
\]
\end{thm}
%-------------------------

%-------------------------
\begin{thm} \label{thm_decay2}
Assume the conditions (a) and (b$_2$) are satisfied. Let $u$ be as above. 
Then we have 
\[
 \|u(t)\|_{L^{2}} \le \frac{C\eps}{\sqrt{1+\eps^2 \log (2+t)}}
\] 
for $t\ge 0$, where $C$ is a positive constant not depending on $\eps$.\\
\end{thm}
%-------------------------

%-------------------------
\begin{thm} \label{thm_asymp_free}
Assume the conditions (a) and (b$_3$) are satisfied. Let $u$ be as above. 
For each $j\in I_N$, there exists $\varphi_j^+ \in L^2(\R_{x})$ with 
$\hat{\varphi}_j^+ \in L^{\infty}(\R_{\xi})$ such that 
\[
 u_j(t)
 =
 e^{i\frac{t}{2m_j}\pa_x^2} \varphi_j^+ + O(t^{-1/4+\delta})
 \quad \mbox{in}\ L^2(\R_x)
\]
and 
\[
 u_j(t,x)=\sqrt{\frac{m_j}{i t}}\, \hat{\varphi}^+_j
 \left( \frac{m_j x}{t}\right) e^{i\frac{m_j x^2}{2t}} 
 + O(t^{-3/4+ \delta})
 \quad \mbox{in}\ L^{\infty}(\R_x) 
\]
as $t \to +\infty$, where $\delta>0$ can be taken arbitrarily small,
and $\hat{\phi}$ denotes the Fourier transform of $\phi$, i.e., 
\[
 \hat{\phi}(\xi)=\frac{1}{\sqrt{2\pi}} \int_{\R} e^{-iy\xi} \phi(y) dy.
\]
\end{thm}
%-------------------------

%------------------
\begin{rmk}
In view of the proof of Theorem \ref{thm_asymp_free} below, 
we can see that $\varphi^+=(\varphi_j^+)_{j \in I_N}$ does not 
identically vanish if the initial data $\varphi$ 
is suitably small and does not identically vanish (see Remark \ref{rmk_opt} 
for the detail). 
Therefore the solution does not gain an additional logarithmic decay 
under the conditions (a) and (b$_3$). \\
\end{rmk}
%-------------------------

Now let us give several examples which satisfy the above
mentioned conditions: 

%------------------
\begin{ex} In the single case (i.e., $N=1$), we may assume $m_1=1$ without 
loss of generality. Then we can check that the condition (a) is euqivalent to 
the gauge invariance \eqref{gi}, and that the condition \eqref{dissip_single} 
is equivalent to the condition (b$_1$). Therefore our results above can be 
viewed as an extension of \cite{HNS} except the explicit asymptotic profile of 
the solution. We can also see that our results cover the system 
\eqref{nls_two} under the assumptions \eqref{mass_resonance}, \eqref{dissip}, 
\eqref{interaction}. Indeed, \eqref{mass_resonance} plays the role of (a), 
and \eqref{dissip}, \eqref{interaction} correspond to (b$_1$) with 
$A=\begin{pmatrix} \kappa_1 & 0\\ 0& \kappa_2\end{pmatrix}$. 
\end{ex}
%-------------------------------

%------------------
\begin{ex}
Next let us consider the following two-component system
\[
 \left\{\begin{array}{l}
 \mathcal{L}_{m} u_1 = \lambda_{1} |u_1|^2u_1  
 + \lambda_{2} \overline{u_1} (\pa_x u_1)^2 
 + i  u_2 \pa_x (\overline{u_1}^2),\\
 \mathcal{L}_{3m} u_2 
 = \lambda_3 |u_2|^2 \pa_x u_2 -i(|u_2|^2 +|\pa_x u_2|^2) u_2 
  -i u_1^2 \pa_x u_1
 \end{array}\right.
\]
with $\lambda_1$, $\lambda_2$, $\lambda_3 \in \C$ and 
$m \in \R\backslash \{0\}$, which is a bit more complicated than 
\eqref{nls_two}. 
It is easy to check that the condition (a) is satisfied by this 
system. Also it follows from simple calculations that 
\[
 \left\{\begin{array}{l}
 p_1(\xi;Y) 
 = (\lambda_{1} -\lambda_{2} m^2 \xi^2) |Y_1|^2 Y_1 
   + 2m \xi \overline{Y_1}^2 Y_2,
 \\
 p_2(\xi;Y) = i(3\lambda_3 m \xi  -1- 9  m^2 \xi^2) |Y_2|^2 Y_2 
 + 3m \xi Y_1^3.
 \end{array}\right.
\]
With $A=\begin{pmatrix} 3 & 0 \\ 0 & 2\end{pmatrix}$, we have 
\[
 \jb{p(\xi;Y), AY}_{\C^2}
 =
 3(\lambda_{1} - \lambda_{2} m^2 \xi^2) |Y_1|^4
 - 2i(1 - 3\lambda_3 m \xi  + 9  m^2 \xi^2) |Y_2|^4
 + 12m \xi \realpart(\overline{Y_1}^3 Y_2),
\]
whence
\[
 \imagpart \jb{p(\xi;Y), AY}_{\C^2}
 =
3\bigl( \imagpart\lambda_{1} - \imagpart \lambda_{2} m^2 \xi^2 \bigr) 
 |Y_1|^4
 - \left\{ 2 - \frac{(\realpart \lambda_3)^2}{2}  
 + 2\biggl(3m\xi -\frac{\realpart \lambda_3 }{2} \biggr)^2 \right\} 
 |Y_2|^4.
\]
Therefore we see that 
\begin{itemize}
\item
(b$_0$) is satisfied if $\imagpart \lambda_1\le 0$, 
$\imagpart \lambda_2\ge 0$ and $|\realpart \lambda_3|\le 2$. 

\item
(b$_1$) is satisfied if $\imagpart \lambda_1 < 0$, $\imagpart \lambda_2\ge 0$ 
and $|\realpart \lambda_3|< 2$. 

\item
(b$_2$) is satisfied if $\imagpart \lambda_1 < 0$, $\imagpart \lambda_2> 0$ 
and $|\realpart \lambda_3|< 2$. 
\end{itemize}
\end{ex}
%--------------------

%----------------------
\begin{ex} Finally we focus on the three-component system 
\[
 \left\{\begin{array}{l}
 \mathcal{L}_{m} u_1 = u_2 \pa_x \bigl( \overline{u_1} u_2\bigr),\\
 \mathcal{L}_{m} u_2 = \overline{u_1}  \overline{u_2}\pa_x u_3 
             + 3 \overline{u_1} u_3 \pa_x \overline{u_2},\\
 \mathcal{L}_{3m} u_3 = 2u_1^2 \pa_x u_2 - u_2\pa_x (u_1^2).
 \end{array}\right.
\]
We can immediately check that this system satisfies (a) and (b$_3$). 
Note that this example should be compared with \cite{IKS}, 
where the null structure in quadratic derivative nonlinear Schr\"odinger 
systems in $\R^2$ is considered in details 
(see also \cite{KatTsu}, \cite{KawSu}, \cite{Su}).\\ 
\end{ex}
%-------------------------------

The rest part of this paper is organized as follows: 
The next section is devoted to preliminaries on basic properties of 
the operator $J_m$. In Scetion \ref{sec_smoothing}, we recall the smoothing 
property of the linear Schr\"odinger eqautions. In Section \ref{sec_apriori}, 
we will get an a priori estimate. After that, The main theorems will be 
proved in Section \ref{sec_pf_of_thm}. 
The appendix is devoted to the proof of technical lemmas. 
In what follows, we will denote several positive constants by the same letter 
$C$, which is possibly different from line to line.

%-------------------------------------------------------------------
\section{Preliminaries}  \label{sec_prelim}
%--------------------------------------------------------------------
In this section, we collect several identities and inequalities 
which are useful for our purpose. We set 
$J_m=x+i\frac{t}{m}\pa_x$ for non-zero real constant $m$. 
Then we can check that $[\pa_x,J_m]=1$ and $[\mathcal{L}_m,J_m]=0$, 
where $[\cdot,\cdot]$ denotes the commutator  of two linear operators. 
We also note that 
\begin{align}
 J_m \phi
 = 
 \frac{it}{m}e^{im\frac{x^2}{2t}} \pa_x \bigl(e^{-im\frac{x^2}{2t}} \phi \bigr),
 \label{rel_j}
\end{align}
which yields the following useful lemmas.
%-----------------------
\begin{lem}\label{lemma_J1}
 Let $m$, $\mu_1$, $\mu_2$, $\mu_3$ be non-zero real constants
satisfying $m = \mu_1 + \mu_2 + \mu_3$. We have
\[
 J_m(f_1f_2f_3) = \frac{\mu_1}{m}(J_{\mu_1}f_1)f_2f_3
                 + \frac{\mu_2}{m}f_1(J_{\mu_2}f_2)f_3
                 + \frac{\mu_3}{m}f_1f_2(J_{\mu_3}f_3),
\]
\[
 J_m(f_1f_2 \overline{f_3}) = \frac{\mu_1}{m}(J_{\mu_1}f_1)f_2\overline{f_3}
                 + \frac{\mu_2}{m}f_1(J_{\mu_2}f_2)\overline{f_3}
                 + \frac{\mu_3}{m}f_1f_2(\overline{J_{-\mu_3}f_3}),
\]
\[
 J_m(f_1 \overline{f_2} \overline{f_3}) 
 = \frac{\mu_1}{m}(J_{\mu_1}f_1) \overline{f_2} \overline{f_3}
   + \frac{\mu_2}{m}f_1 (\overline{J_{-\mu_2}f_2}) \overline{f_3}
   + \frac{\mu_3}{m}f_1 \overline{f_2} (\overline{J_{-\mu_3}f_3}),
\]
\[
 J_m(\overline{f_1} \overline{f_2} \overline{f_3}) 
 =  
  \frac{\mu_1}{m}(\overline{J_{-\mu_1}f_1}) \overline{f_2} \overline{f_3}
 +
  \frac{\mu_2}{m}\overline{f_1} (\overline{J_{-\mu_2}f_2}) \overline{f_3}
 +
  \frac{\mu_3}{m}\overline{f_1} \overline{f_2} (\overline{J_{-\mu_3}f_3})
\]
for smooth $\C$-valued functions $f_1$, $f_2$ and $f_3$. 
\end{lem}
%----------------
\begin{proof} 
We set $\theta=x^2/(2t)$. It follows from \eqref{rel_j} that 
\begin{align*}
 mJ_m(f_1f_2 \overline{f_3}) 
 &=
 it e^{i(\mu_1+\mu_2+\mu_3)\theta} \pa_x \Bigl\{
  (e^{-i\mu_1\theta} f_1)
  (e^{-i\mu_2\theta} f_2)( \overline{e^{i\mu_3\theta} f_3})
 \Bigr\}\\
 &=
 \Bigl( it e^{i\mu_1\theta} \pa_x 
  (e^{-i\mu_1\theta} f_1)\Bigr)
  f_2 \overline{f_3}
 +
  f_1 
 \Bigl( it e^{i\mu_2\theta}\pa_x  (e^{-i\mu_2\theta} f_2)\Bigr)
  \overline{f_3}
 -
 f_1 f_2
 \Bigl(\overline{it e^{-i\mu_3\theta} \pa_x ( e^{i\mu_3\theta} f_3)}\Bigr)
 \\
 &=
 (\mu_1 J_{\mu_1}f_1)f_2\overline{f_3}
 + 
 f_1(\mu_2 J_{\mu_2}f_2)\overline{f_3}
 +
 f_1f_2(\overline{\mu_3 J_{-\mu_3}f_3}),
\end{align*}
which gives the second identity. 
The other three identities can be shown in the same way.\\
\end{proof}

%-------------
\begin{rmk}
 If we do not assume $m=\mu_1+\mu_2+\mu_3$, we have 
 \begin{align*}  
 J_m(f_1f_2f_3) 
 =& \frac{\mu_1}{\mu_1+\mu_2+\mu_3}(J_{\mu_1}f_1)f_2f_3
  + \frac{\mu_2}{\mu_1+\mu_2+\mu_3}f_1(J_{\mu_2}f_2)f_3
  + \frac{\mu_3}{\mu_1+\mu_2+\mu_3}f_1f_2(J_{\mu_3}f_3)\\
 &+ it
    \left(\frac{1}{m} -\frac{1}{\mu_1+\mu_2+\mu_3} \right)
    \pa_x (f_1f_2f_3), 
\end{align*}
and so on. The last term implies a loss of time-decay in general. 
(The situation is worse if $\mu_1+\mu_2+\mu_3=0$.) \\
\end{rmk}
%-------------

%-----------------------
\begin{lem} \label{lemma_J2}
Let $m$, $\mu_1$, $\mu_2$ be non-zero real constants. We have 
\begin{align}
 \pa_x(f_1 f_2 f_3) =\frac{m}{\mu_1}(\pa_x f_1) f_2 f_3
 +\frac{R_1}{t}
 \label{key1}
\end{align}
and 
\begin{align}
 \pa_x^2(f_1 f_2 f_3) 
 =
  \frac{m^2}{\mu_1 \mu_2} (\pa_x f_1)(\pa_x f_2) f_3+ \frac{R_2}{t},
 \label{key2}
\end{align}
where 
$R_1=-im J_m(f_1 f_2 f_3) +im(J_{\mu_1}f_1) f_2 f_3$ and
\[
 R_2=
 -\frac{im^2}{\mu_1} J_{m}\bigl[(\pa_x f_1) f_2 f_3\bigr] 
  +
 \frac{im^2}{\mu_1}(\pa_x f_1)(J_{\mu_2}f_2)f_3
 + 
 \pa_x R_1.
% - im \pa_x J_m(f_1 f_2 f_3) 
% + im \pa_x\bigl[(J_{\mu_1}f_1) f_2 f_3\bigr].
\]
\end{lem}
%-----------------------

%------------
\begin{rmk}
We do not assume any relations among $\mu_1$, $\mu_2$ and $m$ 
in Lemma \ref{lemma_J2}.
\end{rmk}
%------------

\begin{proof} 
From the relation $\frac{1}{m}\pa_x -\frac{1}{it}J_m=i\frac{x}{t}$, 
we see that
\begin{align*}
\frac{1}{m}\pa_x(f_1 f_2 f_3 ) - \frac{1}{it}J_m(f_1 f_2 f_3)
=
i\frac{x}{t} f_1 f_2 f_3
=
\left(\frac{1}{\mu_1}\pa_x f_1 -\frac{1}{it}J_{\mu_1} f_1\right) f_2 f_3,
\end{align*}
which yields \eqref{key1}. 
We also have \eqref{key2} by using \eqref{key1} twice.
\end{proof}

Next we set 
\[
 \bigl( \mathcal{U}_m(t) \phi \bigr)(x)
 :=
 e^{i\frac{t}{2m}\pa_x^2}\phi(x)
 =
 \sqrt{\frac{|m|}{2\pi t}}e^{-i \frac{\pi}{4} \mathrm{sgn}(m)} 
 \int_{\R} e^{im\frac{(x-y)^2}{2t}} \phi(y)dy
\]
for $m \in \R\backslash \{0\}$ and $t>0$. 
We also introduce the scaled Fourier transform $\mathcal{F}_m$ by 
\[
 \bigl( \mathcal{F}_m \phi \bigr)(\xi)
 :=
 |m|^{1/2} e^{-i \frac{\pi}{4} \mathrm{sgn}(m)}\, \hat{\phi}(m\xi)
 =
 \sqrt{\frac{|m|}{2\pi}} e^{-i \frac{\pi}{4} \mathrm{sgn}(m)} 
 \int_{\R} e^{-imy\xi} \phi(y)dy,
\]
as well as auxiliary operators 
\[
 \bigl( \mathcal{M}_m(t) \phi \bigr)(x)
 :=
 e^{im\frac{x^2}{2t}}\phi(x),
 \quad
 \bigl( \mathcal{D}(t)\phi \bigr)(x) 
 := 
 \frac{1}{\sqrt{t}} \phi \left(\frac{x}{t} \right),
 \quad 
  \mathcal{W}_m(t)  \phi 
 :=
 \mathcal{F}_m \mathcal{M}_m (t) \mathcal{F}_m^{-1} \phi,
\]
so that $\mathcal{U}_m$ can be decomposed into 
$\mathcal{U}_m=\mathcal{M}_m \mathcal{D} \mathcal{F}_m \mathcal{M}_m
=\mathcal{M}_m \mathcal{D} \mathcal{W}_m \mathcal{F}_m$. 
The following lemma is well known (see e.g., \cite{HN}, \cite{OS}). 

%-------------------
\begin{lem} \label{lemma_asympt}
Let $m$ be a non-zero real constant. We have 
\begin{align*}
 \|
  \phi - \mathcal{M}_m \mathcal{D} \mathcal{F}_m \mathcal{U}_m^{-1}\phi
 \|_{L^{\infty}}
 \le
 Ct^{-3/4} \bigl(\|\phi\|_{L^2} + \|\mathcal{J}_m \phi\|_{L^2} \bigr)
\end{align*}
and
\[
 \|\phi\|_{L^{\infty}} 
 \le 
 t^{-1/2}\|\mathcal{F}_m \mathcal{U}_m^{-1}\phi\|_{L^{\infty}}
 +Ct^{-3/4} (\|\phi\|_{L^2} + \|\mathcal{J}_m \phi\|_{L^2})
\]
for $t\ge 1$. 
\end{lem}
%------------------

\begin{proof}
For the convenience  of the readers, we give the proof. 
By the relation $J_m=\mathcal{U}_m x \mathcal{U}_m^{-1}$, we see that 
\[
 \|\mathcal{F}_m \mathcal{U}_m^{-1}\phi\|_{H^1}
 \le 
 C \| \mathcal{U}_m^{-1}\phi\|_{H^{0,1}} 
 \le 
  C (\|\phi\|_{L^2} + \|\mathcal{J}_m \phi\|_{L^2}).
\]
Also it follows from the inequalities 
$
\|\phi\|_{L^{\infty}}\le \sqrt{2}\|\phi\|_{L^2}^{1/2}\|\pa_x \phi\|_{L^2}^{1/2}
$
and 
$|e^{i\theta}-1|\le C|\theta|^{1/2}$ that 
\begin{align}
 \|(\mathcal{W}_m^{\pm 1}-1)\phi\|_{L^{\infty}} 
 &\le 
 C \|(\mathcal{M}_m^{\pm 1}-1)\mathcal{F}^{-1}_m \phi\|_{L^2}^{1/2}\, 
  \|\pa_x (\mathcal{W}_m^{\pm 1}-1) \phi\|_{L^2}^{1/2}
 \nonumber\\
 &\le
 C(t^{-1/2} \|\mathcal{F}^{-1}_m \phi\|_{H^{0,1}})^{1/2}\, 
  \|\pa_x \phi\|_{L^2}^{1/2}
 \nonumber\\
 &\le 
 C t^{-1/4}\|\phi\|_{H^1}.
 \label{est_w}
\end{align}
Combining with the inequalities obtained above, we have 
\begin{align*}
 \|
  \phi - \mathcal{M}_m \mathcal{D} \mathcal{F}_m \mathcal{U}_m^{-1}\phi
 \|_{L^{\infty}}
 &= 
 \|
  \mathcal{M}_m \mathcal{D} (\mathcal{W}_m - 1) 
  \mathcal{F}_m \mathcal{U}_m^{-1}\phi
 \|_{L^{\infty}}\\
 &\le
 t^{-1/2} 
 \|(\mathcal{W}_m - 1) \mathcal{F}_m \mathcal{U}_m^{-1}\phi\|_{L^{\infty}}
 \\
 &\le
 Ct^{-3/4} \|\mathcal{F}_m \mathcal{U}_m^{-1}\phi\|_{H^1}\\
 &\le
 Ct^{-3/4} \bigl(\|\phi\|_{L^2} + \|\mathcal{J}_m \phi\|_{L^2} \bigr).
\end{align*}
Using the result derived above, we also have
\begin{align*}
 \|\phi\|_{L^{\infty}}
 &\le 
  \|
  \mathcal{M}_m \mathcal{D} \mathcal{F}_m \mathcal{U}_m^{-1}\phi
 \|_{L^{\infty}}
 +
 \|
  \phi - \mathcal{M}_m \mathcal{D} \mathcal{F}_m \mathcal{U}_m^{-1}\phi
 \|_{L^{\infty}}\\
 &\le 
 t^{-1/2} \|\mathcal{F}_m \mathcal{U}_m^{-1}\phi \|_{L^{\infty}}
 +
 Ct^{-3/4} \bigl(\|\phi\|_{L^2} + \|\mathcal{J}_m \phi\|_{L^2} \bigr).
\end{align*}
\end{proof}

%-------------------
\begin{lem} \label{lemma_prod}
Let $m$ be a non-zero real constant. We have
\[
\|\mathcal{F}_m \mathcal{U}_m^{-1} (f_1 f_2 f_3)\|_{L^{\infty}}
\le 
 C 
 \|f_1\|_{L^2} \|f_2 \|_{L^2} \|f_3\|_{L^{\infty}}.
%or
%\le 
% C t^{-1/2}
% \|f_1\|_{L^2} \|f_2 \|_{L^2} (\|vf_3\|_{L^2}+ \|J_{\mu}f_3\|_{L^2}).
\]
\end{lem}
%------------------

\begin{proof}
By the relation 
$\mathcal{F}_m \mathcal{U}_m^{-1}
 =\mathcal{W}_m^{-1} \mathcal{D}^{-1} \mathcal{M}_m^{-1}$
and the estimate 
$\|\mathcal{W}_m^{-1} \phi \|_{L^{\infty}} \le Ct^{1/2}\|\phi\|_{L^1}$, 
we have 
\begin{align*}
\|\mathcal{F}_m \mathcal{U}_m^{-1} (f_1 f_2 f_3)\|_{L^{\infty}}
&\le 
Ct^{1/2} \| \mathcal{D}^{-1} \mathcal{M}_m^{-1}(f_1f_2f_3) \|_{L^1}\\
&\le
C t^{1/2} \cdot t^{-1} 
\|(\mathcal{D}^{-1} f_1)(\mathcal{D}^{-1} f_2)
(\mathcal{D}^{-1} \mathcal{M}_m^{-1}f_3)\|_{L^1}\\
&\le
 C t^{-1/2} \|\mathcal{D}^{-1} f_1\|_{L^2}
 \|\mathcal{D}^{-1} f_2\|_{L^2} 
 \|\mathcal{D}^{-1} \mathcal{M}_m^{-1}f_3\|_{L^{\infty}}\\
&=
 C t^{-1/2} \|f_1\|_{L^2} \|f_2\|_{L^2} 
 \cdot t^{1/2} \|f_3\|_{L^{\infty}}\\
&= 
  C \|f_1\|_{L^2} \|f_2\|_{L^2} \|f_3\|_{L^{\infty}}.
\end{align*}
\end{proof}

We deduce the following proposition from 
Lemmas \ref{lemma_J1}--\ref{lemma_prod}, 
which will play the key role in Section~\ref{subsec_pointwise}.

%----------------------
\begin{prp} \label{prp_main}
Suppose that the condition (a) is satisfied. For 
a $\C^N$-valued function $u=(u_j(t,x))_{j\in I_N}$, we set 
$
 \alpha_j(t,\xi)=\mathcal{F}_{m_j}[\mathcal{U}_{m_j}(t)^{-1} u_j(t,\cdot)](\xi)
$ and $\alpha=(\alpha_j(t,\xi))_{j \in I_N}$. 
Then we have
\[
 \left\|
 \mathcal{F}_{m_j} \mathcal{U}_{m_j}^{-1} \Bigl[\pa_x^l F_j(u,\pa_x u) \Bigr]
 -
 \frac{(im_j \xi )^l}{t} p_j (\xi; \alpha)
 \right\|_{L^{\infty}_{\xi}}
 \le \frac{C}{t^{5/4}} 
 \sum_{k=1}^N 
 \bigl( 
    \|u_k(t)\|_{H^3} + \|J_{m_k} u_k(t)\|_{H^2} 
 \bigr)^3
\]
for $j \in I_{N}$, $l\in \{0,1,2\}$ and $t\ge 1$.
\end{prp}
%-----------------------
\begin{proof}
For simplicity of exposition, we treat only the case where 
$F_j=(\pa_x u_{1})(\overline{\pa_x u_{2}})(\pa_x u_{3})$ with 
$m_j=m_1-m_2+m_3$. 
The general case can be shown in the same way.

We set 
$\alpha_k^{(s)}=(im_k\xi)^s \alpha_k$ for $s\in \Z_{\ge 0}$, so that 
\[
 \pa_x^s u_k 
 =
 \mathcal{U}_{m_k} \mathcal{F}_{m_k}^{-1} \alpha_k^{(s)}
 =
 \mathcal{M}_{m_k} \mathcal{D}\mathcal{W}_{m_k} \alpha_k^{(s)},
 \qquad
 \pa_x^s \overline{u_k} 
 =
 \mathcal{U}_{-m_k} \mathcal{F}_{-m_k}^{-1} \overline{\alpha_k^{(s)}}.
\]
Remark that 
\[
 p_j(\xi;\alpha)
=(im_1\xi)(-im_2\xi)(im_3\xi)\alpha_1 \overline{\alpha_2} \alpha_3
=\alpha_1^{(1)} \overline{\alpha_2^{(1)}} \alpha_3^{(1)}.
\]

Now we consider the simplest case $l=0$. By the factorization 
of $\mathcal{U}_{m_{j}}$ and the condition $m_{j}=m_{1}-m_{2}+m_{3}$,
we have 
\begin{align*}
 \mathcal{F}_{m_j}\mathcal{U}_{m_j}^{-1} F_j
 &=
 \mathcal{W}_{m_j}^{-1} \mathcal{D}^{-1} \mathcal{M}_{m_j}^{-1}
 \Bigl[
  (\mathcal{M}_{m_1} \mathcal{D}\mathcal{W}_{m_1} \alpha_1^{(1)})
  (\mathcal{M}_{-m_2} \mathcal{D}\mathcal{W}_{-m_2} \overline{\alpha_2^{(1)}})
  (\mathcal{M}_{m_3} \mathcal{D}\mathcal{W}_{m_3} \alpha_3^{(1)})  
 \Bigr]\\
 &=
 \frac{1}{t}\mathcal{W}_{m_j}^{-1} 
 \Bigl[
  (\mathcal{W}_{m_1} \alpha_1^{(1)})
  (\mathcal{W}_{-m_2} \overline{\alpha_2^{(1)}})
  (\mathcal{W}_{m_3} \alpha_3^{(1)})  
 \Bigr]\\ 
 &=
 \frac{1}{t} p_j(\xi;\alpha) +\frac{1}{t}r_0,
\end{align*}
where
\[
 r_0=\mathcal{W}_{m_j}^{-1} 
 \Bigl[
  (\mathcal{W}_{m_1} \alpha_1^{(1)})
  (\mathcal{W}_{-m_2} \overline{\alpha_2^{(1)}})
  (\mathcal{W}_{m_3} \alpha_3^{(1)})  
 \Bigr]
 -
 \alpha_1^{(1)} \overline{\alpha_2^{(1)}} \alpha_3^{(1)}.
\]
Since we can rewrite it as 
\begin{align*}
 r_0
 =&
 (\mathcal{W}_{m_j}^{-1} -1)\Bigl[
  (\mathcal{W}_{m_1} \alpha_1^{(1)})
  (\mathcal{W}_{-m_2} \overline{\alpha_2^{(1)}})
  (\mathcal{W}_{m_3} \alpha_3^{(1)})  
 \Bigr]
 +
 \bigl\{(\mathcal{W}_{m_1} -1)\alpha_1^{(1)}\bigr\}
  (\mathcal{W}_{-m_2} \overline{\alpha_2^{(1)}})
  (\mathcal{W}_{m_3} \alpha_3^{(1)})  \\
 &+
   \alpha_1^{(1)} 
  \bigl\{(\mathcal{W}_{-m_2} -1) \overline{\alpha_2^{(1)}} \bigr\}
  (\mathcal{W}_{m_3} \alpha_3^{(1)})  
 +
   \alpha_1^{(1)} \overline{\alpha_2^{(1)}} 
  \bigl\{(\mathcal{W}_{m_3} -1)\alpha_3^{(1)}\bigr\},
\end{align*}
we can apply \eqref{est_w} and the Sobolev imbedding 
$H^1(\R^1) \hookrightarrow L^{\infty}(\R^1)$ to obtain
\[
 \|r_0\|_{L^{\infty}} \le Ct^{-1/4}\|u_1\|_{H^2} \|u_2\|_{H^2} \|u_3\|_{H^2}.
\]
Next we consider the case of $l=1$. By \eqref{key1} with $m=m_j$, $\mu=m_1$, 
$f_1=\pa_x u_1$, $f_2=\overline{\pa_x u_2}$, $f_3=\pa_x u_3$, we have 
\begin{align}
 \pa_x F_j
 = 
 \frac{m_j}{m_1} (\pa_x^2 u_1)(\overline{\pa_x u_2})(\pa_x u_3) 
 + \frac{R_1}{t},
 \label{rel_1}
\end{align}
where 
\[
 R_1
 =-im_jJ_{m_j} \Bigl[(\pa_x u_1)(\overline{\pa_x u_2})(\pa_x u_3) \Bigr]
  +
  im_j(J_{m_1}\pa_x u_1)(\overline{\pa_x u_2})(\pa_x u_3).
\]
By applying Lemma \ref{lemma_J1} to the first term and using 
Lemma \ref{lemma_prod}, we see that 
\begin{align}
\|\mathcal{F}_{m_j} \mathcal{U}_{m_j}^{-1} R_1 \|_{L^{\infty}}
\le &
 C 
 \|J_{m_1} \pa_x u_1\|_{L^2} \|{\pa_x u_2} \|_{L^2} 
 \|\pa_x u_3\|_{L^{\infty}}
 \nonumber\\
 &+
 C \|\pa_x u_1\|_{L^2}\|{J_{m_2} \pa_x u_2}\|_{L^2} 
 \|\pa_x u_3 \|_{L^{\infty}} 
 \nonumber\\ 
 &+
 C  \|\pa_x u_1 \|_{L^2} \|{\pa_x u_2}\|_{L^{\infty}}
 \|J_{m_3} \pa_x u_3\|_{L^2}
 \nonumber\\
 \le& 
 \frac{C}{t^{1/2}} 
 \sum_{k=1}^{3} \bigl( \|u_k\|_{H^1} + \|J_{m_k} u_k\|_{H^1} \bigr)^3,
 \label{est_R1}
\end{align}
where we have used the inequality 
$\|\phi\|_{L^{\infty}}
 \le Ct^{-1/2}\|\phi\|_{L^2}^{1/2} \|J_m\phi \|_{L^2}^{1/2}$ and 
the commutation relation $[\pa_x,J_m]=1$ in the 
last line. As for the first term of \eqref{rel_1}, similar computations 
as in the previous case lead to 
\begin{align*}
 \mathcal{F}_{m_j} \mathcal{U}_{m_j}^{-1} 
 \Bigl[(\pa_x^2 u_1) ( \overline{\pa_x u_2}) (\pa_x u_3)\Bigr]
  &=
 \frac{1}{t}\mathcal{W}_{m_j}^{-1} 
 \Bigl[
  (\mathcal{W}_{m_1} \alpha_1^{(2)})
  (\mathcal{W}_{-m_2} \overline{\alpha_2^{(1)}})
  (\mathcal{W}_{m_3} \alpha_3^{(1)})  
 \Bigr]\\ 
 &=
 \frac{1}{t} \alpha_1^{(2)} \overline{\alpha_2^{(1)}} \alpha_3^{(1)} 
 + 
 \frac{r_1}{t},
\end{align*}
where
\begin{align*}
r_{1}=\mathcal{W}_{m_j}^{-1} 
 \Bigl[(\mathcal{W}_{m_1} \alpha_1^{(2)})
  (\mathcal{W}_{-m_2} \overline{\alpha_2^{(1)}})
  (\mathcal{W}_{m_3} \alpha_3^{(1)})\Bigr ] 
   -\alpha_1^{(2)} \overline{\alpha_2^{(1)}} \alpha_3^{(1)}. 
\end{align*}
This can be estimated as follows: 
\begin{align*}
 \|r_1\|_{L^{\infty}} 
 \le Ct^{-1/4}\|\pa_x u_1\|_{H^2} \|u_2\|_{H^2} \|u_3\|_{H^2}.
\end{align*}
Moreover, we observe that 
\[
  \frac{im_j\xi}{t} p_j(\xi;\alpha)
 =
 \frac{m_j}{m_1} 
 \frac{im_1\xi}{t} \alpha_1^{(1)} \overline{\alpha_2^{(1)}} \alpha_3^{(1)} 
 =
 \frac{m_j}{m_1} \cdot \frac{1}{t} 
 \alpha_1^{(2)} \overline{\alpha_2^{(1)}} \alpha_3^{(1)}. 
\]
Piecing them together, we arrive at 
\begin{align*}
 \left\| 
  \mathcal{F}_{m_j} \mathcal{U}_{m_j}^{-1}\pa_x F_j 
 -\frac{im_j\xi}{t}p_j(\xi;\alpha) \right\|_{L^{\infty}_{\xi}}
 &=
 \frac{1}{t}
 \left\|
 \frac{m_j}{m_1} r_1 + \mathcal{F}_{m_j} \mathcal{U}_{m_j}^{-1} R_1
 \right\|_{L^{\infty}}\\
 &\le
 \frac{C}{t} \bigl(\|r_1\|_{L^{\infty}}
 +
 \|\mathcal{F}_{m_j} \mathcal{U}_{m_j}^{-1} R_1\|_{L^{\infty}} \bigr)\\
 &\le
 \frac{C}{t^{5/4}} 
  \sum_{k=1}^3 
 \bigl( 
    \|u_k(t,\cdot)\|_{H^3} + \|J_{m_k} u_k(t,\cdot)\|_{H^1} 
 \bigr)^3,
\end{align*}
as desired.
Finally we consider the case of $l=2$. 
By \eqref{key2} with $m=m_j$, $\mu_1=m_1$ and $\mu_2=-m_2$, we have 
\[
 \pa_x^2 F_j
 = 
 \frac{m_j^2}{-m_1m_2} (\pa_x^2 u_1)(\overline{\pa_x^2 u_2})(\pa_x u_3) 
 +\frac{R_2}{t},
\]
where
\begin{align*}
 R_2=&
 -\frac{im_j^2}{ m_1}
  J_{m_j}\Bigl[(\pa_x^2 u_1) (\overline{\pa_x u_2}) (\pa_x u_3) \Bigr]
 +\frac{im_j^2}{ m_1}(\pa_x^2 u_1)(\overline{J_{m_2} \pa_x u_2}) (\pa_x u_3)
 \\
 &-i m_j \pa_x J_{m_j} 
 \bigl[(\pa_x u_1) (\overline{\pa_x u_2}) (\pa_x u_3)\bigr] 
 +i m_j \pa_x\left[(J_{m_1}\pa_x u_1) (\overline{\pa_x u_2}) (\pa_x u_3)
 \right].
\end{align*}
As in the derivation of \eqref{est_R1}, we see that 
\[
 \|\mathcal{F}_{m_j} \mathcal{U}_{m_j}^{-1} R_2 \|_{L^{\infty}}
\le 
 \frac{C}{t^{1/2}} 
 \sum_{k=1}^{3} (\|u_k\|_{H^2} + \|J_{m_k} u_k\|_{H^2})^3.
\]
Similarly to the previous cases, we can also show that 
\begin{align*}
 \mathcal{F}_{m_j} \mathcal{U}_{m_j}^{-1} 
 \Bigl[(\pa_x^2 u_1)(\overline{\pa_x^2 u_2})(\pa_x u_3)\Bigr]
  &=
 \frac{1}{t}\mathcal{W}_{m_j}^{-1} 
 \Bigl[
  (\mathcal{W}_{m_1} \alpha_1^{(2)})
  (\mathcal{W}_{-m_2} \overline{\alpha_2^{(2)}})
  (\mathcal{W}_{m_3} \alpha_3^{(1)})  
 \Bigr]\\ 
 &=
 \frac{-m_1m_2}{m_j^2} \frac{(im_j\xi)^2}{t} p_j(\xi;\alpha)  
 + 
 \frac{r_2}{t},
\end{align*}
where
\begin{align*}
r_{2}=\mathcal{W}_{m_j}^{-1} 
 \Bigl[
  (\mathcal{W}_{m_1} \alpha_1^{(2)})
  (\mathcal{W}_{-m_2} \overline{\alpha_2^{(2)}})
  (\mathcal{W}_{m_3} \alpha_3^{(1)})  
 \Bigr]-
 \alpha_1^{(2)} \overline{\alpha_2^{(2)}} \alpha_3^{(1)}.
\end{align*}
Note that 
\[
 \|r_2\|_{L^{\infty}} 
 \le Ct^{-1/4}\|\pa_x u_1\|_{H^2} \|\pa_x u_2\|_{H^2} \|u_3\|_{H^2}.
\]
Therefore we have 
\begin{align*}
  \left\|
   \mathcal{F}_{m_j} \mathcal{U}_{m_j}^{-1}\pa_x^2 F_j 
 -\frac{(im_j\xi)^2}{t}p_j(\xi,\alpha)
  \right\|_{L^{\infty}_{\xi}}
 &\le
 \frac{C}{t} \bigl(\|r_2\|_{L^{\infty}}
 +
 \|\mathcal{F}_{m_j} \mathcal{U}_{m_j}^{-1} R_2\|_{L^{\infty}} \bigr)\\
 &\le
 \frac{C}{t^{5/4}} 
  \sum_{k=1}^3 
 \bigl( 
    \|u_k(t,\cdot)\|_{H^3} + \|J_{m_k} u_k(t,\cdot)\|_{H^2} 
 \bigr)^3,
\end{align*}
which completes the proof.
\end{proof}

%-------------------------------------------------------------------
\section{Smoothing effect}  \label{sec_smoothing}
%--------------------------------------------------------------------
In this section, we recall smoothing properties of the linear Schr\"odinger 
equations. As is well known, the standard energy method causes a derivative 
loss when the nonlinear term involves derivatives of the unknown functions. 
Smoothing effect is a useful tool to overcome this obstacle. 
Among various kinds of such techniques, we will follow the approach of 
\cite{HNP}. Let $\mathcal{H}$ be the Hilbert transform, that is, 
\[
 \mathcal{H}\psi(x):=
 \frac{1}{\pi} \, \mathrm{p.v.}\int_{\R} \frac{\psi(y)}{x-y}dy.
\]
With a non-negative weight function $\Phi(x)$ and a non-zero real constant 
$m$, let us also define the operator $S_{\Phi,m}$ by
\[
 S_{\Phi,m}\psi (x):=
 \left\{ \cosh\biggl(\int_{-\infty}^{x} \Phi(y) dy\biggr) \right\}
 \psi(x)
 - i\, \mathrm{sgn}(m)
 \left\{\sinh\biggl(\int_{-\infty}^{x} \Phi(y) dy\biggr) \right\}
 \mathcal{H}\psi(x).
\]
Note that $S_{\Phi,m}$ is $L^2$-automorphism and that both 
$\|S_{\Phi,m}\|_{L^2\to L^2}$, $\|S_{\Phi,m}^{-1}\|_{L^2\to L^2}$ 
are dominated by $C\exp(\|\Phi\|_{L^1})$. This operator enables us to 
gain the half-derivative $|\pa_x|^{1/2}$. 
More precisely, we have the following:

%-----------------------
\begin{lem} \label{lemma_smoothing}
Let $m$, $\mu_1,\ldots , \mu_N$ be non-zero real constants. 
Let $v$ be a $\C$-valued smooth function of $(t,x)$, and 
let $w=(w_j)_{j \in I_N}$ be a $\C^N$-valued smooth function of $(t,x)$. 
We set $\Phi=\eta(|w|^{2}+|\partial_{x}w|^2)$ with $\eta\ge 1$, and 
$S=S_{\Phi(t,\cdot),m}$. 
Then we have 
\begin{align*}
 \frac{d}{dt}\|Sv(t)\|_{L^2}^2 
 +&
 \frac{1}{|m|}\int_{\R} \Phi(t,x) \Bigl| S|\pa_x|^{1/2}  v(t,x) \Bigr|^2 dx\\
 &\le
 2 \Bigl| \Jb{Sv(t), S\mathcal{L}_m v(t)}_{L^2}\Bigr| 
 + CB(t) \|v(t)\|_{L^2}^2,
\end{align*}
where
\begin{align*}
 B(t)
 =
 e^{C\eta \|w\|_{H^1}^2}
 \left\{
  \eta \|w(t)\|_{W^{2,\infty}}^2 + \eta^{3} \|w(t)\|_{W^{1,\infty}}^6
  +
  \eta  \sum_{k \in I_N}\|w_k(t)\|_{H^1} 
  \|\mathcal{L}_{\mu_k} w_k(t)\|_{H^1}
 \right\}
\end{align*}
and the constant $C$ is independent of $\eta$. 
We denote by $W^{s,\infty}$ the $L^{\infty}$-based Sobolev space of 
order $s \in \Z_{\ge 0}$.
\\
\end{lem} 
%-----------------------

This lemma is essentially the same as Lemma~{2.1} in \cite{HNP}, although 
we need slight modifications to fit for our purpose. For the convenience of 
the readers, we will give the proof of this lemma in the appendix.

By using Lemma \ref{lemma_smoothing} combined with the following 
auxiliary lemma, 
we can get rid of the derivative loss coming from the nonlinear terms.

%-----------------------
\begin{lem} \label{lemma_aux2}
Let $m_1,\ldots, m_N$ be non-zero real constants. 
Let $v=(v_j)_{j\in I_N}$, $w=(w_j)_{j \in I_N}$ be 
$\C^N$-valued smooth functions of $x \in \R$. 
Suppose that $q_{1,jk}$ and $q_{2,jk}$  are quadratic homogeneous 
polynomials in $(w,\pa_x w, \overline{w}, \overline{\pa_x w})$. 
We set $\Phi=\eta(|w|^{2}+|\partial_{x}w|^2)$ with $\eta\ge 1$, and 
$S=S_{\Phi(t,\cdot),m}$ with $\eta\ge 1$, and 
$S_j=S_{\Phi, m_j}$ for $j \in I_N$. Then we have
\begin{align*}
 \sum_{j,k\in I_N}
 &\biggl(
  \left|\Jb{S_j v_j, S_j \bigl( q_{1,jk}\pa_x v_k \bigr)}_{L^2} \right|
 + 
 \left|
  \Jb{S_j v_j, S_j \bigl( q_{2,jk}\overline{\pa_x v_k} \bigr)}_{L^2} 
 \right| \biggr)\\
 &\le
 \frac{C}{\eta} e^{C\eta \|w\|_{H^1}^2} 
  \sum_{k \in I_N} 
  \int_{\R} \Phi(x) \Bigl| S_k |\pa_x|^{1/2} v_k(x) \Bigr|^2 dx\\
 &\hspace{4mm}+
 Ce^{C \eta \|w\|_{H^1}^2} 
 \bigl( 1+ \eta^{2}\|w\|_{H^1}^4+\eta^{2}\|w\|_{W^{1,\infty}}^4 \bigr) 
 \|w\|_{W^{2,\infty}}^2 \|v\|_{L^2}^2,
\end{align*}
where the constant $C$ is independent of $\eta$.\\
\end{lem} 
%-----------------------

We skip the proof of Lemma~\ref{lemma_aux2} because this is nothing more 
than a paraphrase of Lemma~{2.3} in \cite{HNP}.

%-------------------------------------------------------------------
\section{A priori estimate}  \label{sec_apriori}
%--------------------------------------------------------------------

Let $T\in (0,+\infty]$, and 
let $u=(u_j)_{1\le j \le N}\in C([0,T);H^3\cap H^{2,1})$ be 
a solution to \eqref{nls_system} for $t\in [0,T)$. 
As in Section \ref{sec_prelim}, we set 
$\alpha_j(t,\xi)=
\mathcal{F}_{m_j} \Bigl[ \mathcal{U}_{m_j}^{-1} u_j(t,\cdot) \Bigr](\xi)$, 
$\alpha(t,\xi)=(\alpha_j(t,\xi))_{j \in I_N}$, and define
\begin{align*}
 E(T) =\sup_{0\le t< T}
 \sum_{j\in I_N}\biggl[
 (1+t)^{-\frac{\gamma}{3}} \Bigl(
  \|u_j(t)\|_{H^{3}} + \|J_{m_j} u_j(t)\|_{H^{2}}
 \Bigr) 
 + 
 \sup_{\xi \in \R} \Bigl(\jb{\xi}^2 |\alpha_j(t,\xi)|\Bigr)
 \biggr]
\end{align*}
with $\gamma >0$. The goal of this section is to show the following: 

%---------------
\begin{lem} \label{lem_apriori}
Assume the conditions (a) and (b$_0$) are satisfied. 
Let $\gamma  \in (0,1/4)$. 
There exist positive constants 
$\eps_{1}$ and $K$ such that 
\begin{align}
 E(T)\le \eps^{2/3}
 \label{est_before}
\end{align}
implies  
\begin{align*}
 E(T) \le K\eps,
 %\label{est_after}
\end{align*}
provided that $\eps=\|\varphi\|_{H^{3}\cap H^{2,1}}\le \eps_{1}$.
\end{lem}
%----------------

The proof of this lemma will be divided into two parts. 

%-----------------------------------------------
\subsection{$L^2$-estimates}  \label{subsec_L2}
%-----------------------------------------------

In the first part, we consider the bounds for $\|u_j(t)\|_{H^3}$ and 
$\|J_{m_j}u_j(t)\|_{H^2}$. It is enough to show
\begin{align}
 \sum_{j\in I_N}\sum_{l=0}^{1} \|J_{m_j}^l u_j(t)\|_{L^2}
 \le 
 C\eps + C\eps^2(1+t)^{\gamma/3}
 \label{est_lower}
\end{align}
and 
\begin{align}
 \sum_{j \in I_N}\sum_{l=0}^{1}\|\pa_x^{3-l} J_{m_j}^l u_j(t)\|_{L^2}^2
 \le 
 C\eps^2 (1+t)^{2\gamma/3}
 \label{est_high}
\end{align}
for $t\in [0,T)$ under the assumption \eqref{est_before}. 
First we remark that \eqref{est_before} implies a rough 
$H^1$-bound
\begin{align}
 \|u_j(t)\|_{H^1} 
 \le
 C\|\alpha_j(t)\|_{H^{0,1}}
 \le
 C\left(\int_{\R} \frac{d\xi}{\jb{\xi}^2}\right)^{1/2} 
 \sup_{\xi \in \R} \left( \jb{\xi}^2 |\alpha_j(t,\xi)|\right)
 \le 
 C\eps^{2/3}
 \label{est_rough1}
\end{align}
for $t \in [0,T)$.
We also deduce from \eqref{est_before} that 
\begin{align*}
 \|u_j(t)\|_{W^{2,\infty}} \le \frac{C\eps^{2/3}}{(1+t)^{1/2}}
\end{align*}
for $t \in [0,T)$. Indeed, it follows from Lemma \ref{lemma_asympt}
and the relation $[\partial_{x}, J_{m_{j}}]=1$ that 
\begin{align*}
 \|u_j(t)\|_{W^{2,\infty}} 
 \le 
 \frac{C}{t^{1/2}}\sup_{\xi\in \R}|\jb{\xi}^2\alpha_j(t,\xi)| 
 + \frac{C}{t^{3/4}} 
  \bigl( \|u_j(t)\|_{H^2} + \|J_{m_j} u_j(t)\|_{H^2} \bigr)
 \le 
 \frac{C\eps^{2/3}}{t^{1/2}}
\end{align*}
for $t\ge 1$, and $H^1(\R^1) \hookrightarrow L^{\infty}(\R^1)$ yields 
$\|u_j(t)\|_{W^{2,\infty}} \le C\|u_j(t)\|_{H^3} \le C\eps^{2/3}$ 
for $t \le 1$.

Now we consider the easier estimate \eqref{est_lower}. It follows from the 
standard energy method that 
\begin{align*}
 \frac{d}{dt} \|u_j(t)\|_{L^2}
 &\le 
 \|F_j(u(t), \pa_x u(t))\|_{L^2}\\
 &\le 
 C\|u(t)\|_{W^{1,\infty}}^2 \|u(t)\|_{H^1}\\
 &\le 
 C\left(\frac{\eps^{2/3}}{(1+t)^{1/2}}\right)^2 \cdot C\eps^{2/3}\\
 &\le 
 \frac{C\eps^{2}}{1+t}.
\end{align*}
Also we see from Lemma \ref{lemma_J1} that 
\[
 \mathcal{L}_{m_j}J_{m_j} u_j
 = 
 \sum_{k \in I_N} \Bigl(
 q_{1,jk} J_{m_k}\pa_x u_{k}
 +
 q_{2,jk} \overline{J_{m_k} \pa_x u_k} 
 + 
 q_{3,jk} J_{m_k}u_{k}
 +
 q_{4,jk} \overline{J_{m_k} u_k} 
 \Bigr),
\]
where $q_{1,jk}, \ldots, q_{4,jk}$ are quadratic homogeneous 
polynomials in $(u,\pa_x u, \overline{u}, \overline{\pa_x u})$. 
Then the standard energy method again implies
\[
 \frac{d}{dt} \|J_{m_j}u_j(t)\|_{L^2} 
 \le 
 C\|u\|_{W^{1,\infty}}^2 \sum_{k \in I_N} 
 (\|u_{k}\|_{H^1}+ \|J_{m_k} u_k\|_{H^1})
 \le 
 \frac{C\eps^{2}}{(1+t)^{1-\gamma/3}}.
\]
These lead to \eqref{est_lower}.

Next we consider \eqref{est_high}. We set 
$v_{jl}=\pa_x^{3-l} J_{m_j}^l u_j$ for $l\in \{0,1\}$ and $j\in I_N$. 
We  apply Lemma \ref{lemma_smoothing} with 
$m=m_{j}$, $\mu_k=m_k$, $v=v_{jl}$, $w=u$, $\eta=\eps^{-2/3}$. 
Then we obtain
\begin{align}
 &\frac{d}{dt} \|S_j v_{jl}(t)\|_{L^2}^2
 + 
 \frac{1}{|m_j|}
 \int_{\R} \Phi(t,x) \Bigl| S_j |\pa_x|^{1/2} v_{jl}(t)\Bigr|^2dx
 \nonumber\\
&\le
 2\left|
 \Jb{S_j v_{jl}, S_j \pa_x^{3-l} J_{m_j}^l F_j(u,\pa_x u)}_{L^2}
 \right|
 + 
CB(t) \|v_{jl}(t)\|_{L^2}^2,
 \label{est_smoothing_a}
\end{align}
where 
\begin{align*}
 B(t)
 &= 
 e^{\frac{C}{\eps^{2/3}} \|u\|_{H^1}^2}
 \left(
  \eps^{-2/3}\|u\|_{W^{2,\infty}}^2 + \eps^{-2} \|u\|_{W^{1,\infty}}^6
  +
  \eps^{-2/3}\sum_{k \in I_N}\|u_k\|_{H^1} \|F_k(u,\pa_x u)\|_{H^1}
 \right)\\
 &\le 
 \frac{C\eps^{2/3}}{1+t}.
\end{align*}
To estimate the first term of the right-hand side of \eqref{est_smoothing_a}, 
we use Lemma \ref{lemma_J1} and the usual Leibniz rule to split 
$\pa_x^{3-l} J_{m_j}^l F_j(u,\pa_x u)$ into the following form: 
\[
 \sum_{k \in I_N} \Bigl(
 g_{1,jkl} \pa_x v_{kl}
 +
 g_{2,jkl} \overline{\pa_x v_{kl}} 
 \Bigr)
 +
 h_{jl},
\]
where $g_{1,jkl}$ and $g_{2,jkl}$ are quadratic homogeneous 
polynomials in $(u,\pa_x u, \overline{u}, \overline{\pa_x u})$, and $h_{jl}$ 
is a cubic term satisfying 
\begin{align*}
 \|h_{jl}\|_{L^2} 
 \le 
 C\|u(t)\|_{W^{2,\infty}}^2 
 \sum_{k \in I_N}(\|u_k(t)\|_{H^3}+\|J_{m_k} u_k(t)\|_{H^2})
 \le
 \frac{C\eps^{2}}{(1+t)^{1-\gamma/3}}.
\end{align*}
Then Lemma \ref{lemma_aux2} and the $L^{2}$-automorphism
of $S_{j}$ lead to 
\begin{align*}
 &\sum_{j \in I_N}
  \left|\Jb{S_jv_{jl}, S_j\pa_x^{3-l} J_{m_j}^l F_j(u,\pa_x u)}_{L^2} \right|
 \\
 &\le
 \sum_{j,k \in I_N}
 \biggl(
  \left|\Jb{S_j v_{jl}, S_j \bigl(g_{1,jkl}\pa_x v_{kl} \bigr)}_{L^2} \right|
 + 
 \left|
  \Jb{S_j v_{jl}, S_j \bigl(g_{2,jkl}\overline{\pa_x v_{kl}} \bigr)}_{L^2} 
 \right| \biggr)
 + \sum_{j \in I_N}\|S_j v_{jl}\|_{L^2} \|S_j h_{jl}\|_{L^2}\\
 &\le
 C\eps^{2/3} e^{\frac{C}{\eps^{2/3}}\|u\|_{H^1}^2} 
  \sum_{k \in I_N}\int_{\R} 
   \Phi(t,x) \Bigl| S_k|\pa_x|^{1/2} v_{kl}(t,x) \Bigr|^2 dx
 \\
 &\hspace{6mm}+
 Ce^{\frac{C}{\eps^{2/3}}\|u\|_{H^1}^2} 
 \bigl( 1+ \eps^{-4/3}\|u\|_{H^1}^4 +\eps^{-4/3}\|u\|_{W^{1,\infty}}^4 \bigr) 
 \|u\|_{W^{2,\infty}}^2 \sum_{k \in I_N}\|v_{kl}\|_{L^2}^2\\
 &\hspace{6mm}+
 Ce^{\frac{C}{\eps^{2/3}}\|u\|_{H^1}^2} \sum_{j \in I_N}
 \|v_{jl}\|_{L^2} \|h_{jl}\|_{L^2}\\
 &\le
  C_{0}\eps^{2/3} 
 \sum_{k \in I_N} \int_{\R} 
  \Phi(t,x) \left| S_k |\pa_x|^{1/2} v_{kl}(t,x)\right|^2dx
  +
  \frac{C\eps^{8/3}}{(1+t)^{1-2\gamma/3}}
\end{align*}
with some positive constant $C_0$ not depending on $\eps$. 
Summing up, we obtain 
\begin{align*}
 \frac{d}{dt} \sum_{j \in I_N} \|S_j v_{jl}(t)\|_{L^2}^2
 &\le 
 \sum_{k \in I_N} \left( 2C_0\eps^{2/3}- \frac{1}{|m_k|} \right)
 \int_{\R} \Phi(t,x) \Bigl| S_k |\pa_x|^{1/2} v_{kl}(t,x)\Bigr|^2dx\\
 &\hspace{6mm}+ 
  \frac{C\eps^{8/3}}{(1+t)^{1-2\gamma/3}}
 + 
  \frac{C\eps^{2/3}}{1+t} \cdot 
 \bigl(C\eps^{2/3} (1+t)^{\gamma/3} \bigr)^2\\
 &\le 
 \frac{C\eps^{2}}{(1+t)^{1-2\gamma/3}},
\end{align*}
provided that 
\[
 2C_0\eps^{2/3}\le  \frac{1}{\dis{\min_{1\le k \le N} |m_k|}}. 
\]
Integrating with respect to $t$, we have 
\[
 \sum_{j \in I_N} \|S_j v_{jl}(t)\|_{L^2}^2
 \le 
 C\eps^2 + C\eps^2 (1+t)^{2\gamma/3}
 \le 
 C\eps^2 (1+t)^{2\gamma/3},
\]
whence
\[
 \sum_{j \in I_N} \sum_{l=0}^{1}\|\pa_x^{3-l} J_{m_j}^l u_j(t)\|_{L^2}^2
 \le 
e^{C\eps^{-2/3}\|u(t)\|_{H^1}^2}
\sum_{j \in I_N}\sum_{l=0}^{1}\|S_j v_{jl}(t)\|_{L^2}^2
 \le 
C\eps^2(1+t)^{2\gamma/3},
\]
as required. 
\qed
%----------------------------------------------------------
\subsection{Estimates for $\alpha_j$}  \label{subsec_pointwise}
%----------------------------------------------------------
In the second part, we are going to show 
$\jb{\xi}^2 |\alpha (t,\xi)| \le C\eps$ 
for $(t,\xi) \in [0,T)\times \R$  
under the assumption \eqref{est_before}. 
If $t\in [0,1]$, the Sobolev imbedding yields 
this estimate immediately. 
Hence we have only to consider the case of $t \in [1,T)$. 
We set 
\[
 \rho_j(t,\xi)=
\mathcal{F}_{m_j} \mathcal{U}_{m_j}^{-1} \big[F_j(u,\pa_x u) \bigr]
-\frac{1}{t}p_j(\xi;\alpha(t,\xi))
\]
and $\rho=(\rho_j)_{j \in I_N}$, so that  
\begin{align}
 i\pa_t \alpha_j(t,\xi)
 &=
 \mathcal{F}_{m_j} \mathcal{U}_{m_j}^{-1}
 \bigl[ \mathcal{L}_{m_j} u_j \bigr]
 =
 \mathcal{F}_{m_j} \mathcal{U}_{m_j}^{-1}\bigl[ F_j(u,\pa_x u) \bigr]
 \nonumber\\
 &=
 \frac{1}{t}p_j(\xi;\alpha(t,\xi)) +\rho_j(t,\xi).
 \label{profile}
\end{align}
By Proposition \ref{prp_main}, we have 
\begin{align*}
|\rho_j(t,\xi)|
 &\le 
 \frac{C}{\jb{\xi}^2}\sum_{l=0}^{2}
 \bigl| (im_j\xi)^l \rho_j(t,\xi) \bigr|\\
 &=
  \frac{C}{\jb{\xi}^2}\sum_{l=0}^{2}
 \left| 
  \mathcal{F}_{m_j} \mathcal{U}_{m_j}^{-1} \big[\pa_x^l F_j(u,\pa_x u) \bigr]
  -\frac{(im_j\xi)^l}{t}p_j(\xi;\alpha(t,\xi))
 \right|\\
 &\le
   \frac{C}{\jb{\xi}^2} \cdot \frac{C}{t^{5/4}} \left(E (T)
          t^{\frac{\gamma}{3}}\right)^3\\
 &\le
  \frac{C \eps^2}{\jb{\xi}^2 t^{5/4-\gamma}}
\end{align*}
for $t\geq 1$ and $\xi \in \R$, 
which shows that $\rho_j(t,\xi)$ has enough decay rates both in 
$t$ and $\xi$. 
Now we put $\nu(t,\xi)=\sqrt{\jb{\alpha(t,\xi), A\alpha(t,\xi)}_{\C^N}}$, 
where $A$ is the positive Hermitian matrix appearing in the condition (b$_0$). 
Remark that 
\[
 \sqrt{\kappa_*} |\alpha(t,\xi)| 
 \le 
 \nu(t,\xi) 
 \le 
 \sqrt{\kappa^*} |\alpha(t,\xi)|,
\]
where $\kappa_*$ and $\kappa^*$ are the smallest and largest 
eigenvalues of $A$, respectively. 
It follows from (b$_0$) that 
\begin{align*}
 \pa_t \nu(t,\xi)^2
 &=
 2\imagpart \jb{i\pa_t \alpha(t,\xi), A\alpha(t,\xi)}_{\C^N}\\
 &=
 \frac{2}{t}\imagpart \jb{p(\xi;\alpha(t,\xi)), A\alpha(t,\xi)}_{\C^N}
 +
 2\imagpart  \jb{\rho(t,\xi), A\alpha(t,\xi)}_{\C^N}\\
 &\le
 0+ C|\rho(t,\xi)| \nu(t,\xi),
\end{align*}
which leads to 
\begin{align*}
 \nu(t,\xi)
 \le 
  \nu(1,\xi) + {C}\int_1^t |\rho(\tau,\xi)| d\tau
 \le
 \frac{C\eps}{\jb{\xi}^2} 
 + 
 \frac{C\eps^2}{\jb{\xi}^2} 
  \int_1^{\infty}\frac{d\tau}{\tau^{5/4-\gamma}}
 \le
 \frac{C\eps}{\jb{\xi}^2},
\end{align*}
Therefore we have
\[
 %\sum_{l=0}^{2}|(im_j\xi)^l \alpha_j(t,\xi)|
 \jb{\xi}^2 |\alpha_j(t,\xi)|
 \le 
 C\jb{\xi}^2 \nu(t,\xi)
 \le
 C\eps,
\]
as required.\qed

%-------------------------------------------------------------------
\section{Proof of the main theorems}  \label{sec_pf_of_thm}
%--------------------------------------------------------------------

Now we are in a position to prove Theorems \ref{thm_sdge} -- 
\ref{thm_asymp_free}.

%-----------------------------------------------
\subsection{Proof of Theorem \ref{thm_sdge}}  \label{proof_sdge}
%-----------------------------------------------

First let us recall the local existence theorem. 
For fixed  $t_0\ge 0$, let us consider the initial value 
problem 
\begin{align}
\left\{\begin{array}{cl}
\mathcal{L}_{m_j} u_j=F_j(u,\pa_x u), & t>t_0,\ x\in \R,\ j \in I_N,\\
u_j(t_0,x)=\psi_j(x), & x \in \R,\ j \in I_N.
\end{array}\right.
\label{ivp_shift}
\end{align}

%----------------------
\begin{lem} \label{lem_local}
Let  $\psi=(\psi_j)_{j\in I_N} \in H^3\cap H^{2,1}$. 
There exists a positive constant $\eps_0$, which is independent of $t_0$, 
such that the following holds: 
for any $\underline{\eps} \in (0,\eps_0)$ and $M\in (0,\infty)$, 
one can choose a positive constant $\tau^*=\tau^{*}(\underline{\eps},M)$, 
which is independent of $t_0$, 
such that \eqref{ivp_shift} admits a unique solution 
$u=(u_j)_{j\in I_N} \in C([t_0,t_0+\tau^*]; H^3\cap H^{2,1})$, 
provided that 
\[
 \|\psi\|_{H^{1}}  \le \underline{\eps}
 \quad  \mbox{and}\quad  
 \sum_{l=0}^{1} \sum_{j \in I_N}
 \Bigl\|\bigl(x+i\frac{t_0}{m_j}\pa_x \bigr)^l\psi_j \Bigr\|_{H^{3-l}} 
 \le M. 
\]
\end{lem}
%-----------------------
We omit the proof of this lemma because it is standard 
(see e.g., Appendix of \cite{IKS} for the proof of similar lemma 
in the quadratic nonlinear case).
\\

Now we are going to prove the global existence by the so-called 
bootstrap argument. 
Let $T^*$ be the supremum of all $T \in (0,\infty]$ such that 
the problem \eqref{nls_system} admits a unique solution 
$u \in C([0,T);H^3\cap H^{2,1})$. 
By Lemma \ref{lem_local} with $t_0=0$, we have $T^*>0$ if 
$\|\varphi\|_{H^1}\le \eps< \eps_0$. 
We also set 
\[
 T_{*}=\sup \bigl\{ \tau \in [0,T^*)\, |\, E(\tau)\le \eps^{2/3} \bigr\}.
\]
Note that $T_*>0$ because of the continuity of 
$[0,T^*) \ni \tau \mapsto E(\tau)$ and 
$\|\varphi\|_{H^3\cap H^{2,1}}=\eps\le \frac{1}{2}\eps^{2/3}$ 
if $\eps\le 1/8$. 

We claim that $T_*=T^{*}$ if $\eps$ is small enough. Indeed, if 
$T_{*}<T^*$, Lemma \ref{lem_apriori} with $T=T_*$ yields 
\[
 E(T_*) \le K\eps \le \frac{1}{2}\eps^{2/3}
\] 
for $\eps\le \eps_2:=\min\{\eps_1, 1/(2K)^3\}$, 
where $K$ and $\eps_1$ are mentioned in Lemma \ref{lem_apriori}. 
By the continuity of $[0,T^*)\ni \tau \mapsto E(\tau)$, we can take 
$T^{\flat} \in (T_*,T^*)$ such that 
$E(T^{\flat})\le \eps^{2/3}$, which contradicts the 
definition of $T_*$. Therefore we must have $T_*=T^*$. 
By using Lemma \ref{lem_apriori} with $T=T^*$ again, we see that 
\[
  \sum_{l=0}^{1}\sum_{j \in I_N}\|J_{m_j}^{l} u_j(t,\cdot)\|_{H^{3-l}}
  \le 
  K\eps (1+t)^{\frac{\gamma}{3}}, 
 \qquad  
 \sum_{j \in I_N}\sup_{\xi \in \R} \Bigl(\jb{\xi}^2 |\alpha_j(t,\xi)|\Bigr)
 \le K\eps
\]
for $t\in [0,T^*)$. In particular we have
\[
 \sup_{t\in [0,T^*)}\|u(t)\|_{H^1} 
 \le 
  C\sup_{(t,\xi)\in [0,T^*) \times \R} \Bigl(\jb{\xi}^2 |\alpha(t,\xi)|\Bigr)
 \le 
 C^{\flat}\eps
\]
with some $C^{\flat}>0$. 

Next we assume $T^*<\infty$. Then, by setting 
$\eps_3=\min\{\eps_2, \eps_0/2C^{\flat}\}$ 
and $M=K\eps_3 (1+T^*)^{\gamma/3}$, we have 
\[
  \sup_{t\in [0,T^*)} \sum_{l=0}^{1} \sum_{j\in I_N} 
 \|J_{m_j}^{l} u_j(t,\cdot)\|_{H^{3-l}} \le M
\]
as well as 
\[
 \sup_{t\in [0,T^*)}\|u(t)\|_{H^1} \le \eps_0/2<\eps_0
\]
for $\eps \le \eps_3$. By Lemma \ref{lem_local}, 
there exists $\tau^*>0$ such that \eqref{nls_system} admits the solution 
$u\in C([0, T^*+\tau^*);H^{3}\cap H^{2,1})$. This contradicts the definition 
of $T^*$, which means $T^*=+\infty$ for $\eps \in (0,\eps_3]$. 
Moreover, we have
\[
 \|u(t)\|_{L^2} \le C\sup_{\xi \in \R} |\jb{\xi}\alpha(t,\xi)|
 \le C\eps.
\]
By using Lemma \ref{lemma_asympt} and the inequality obtained above,
we also have 
\[
 |u_j(t,x)| \le \frac{C}{t^{1/2}} |\alpha_j(t,\xi)| +\frac{C}{t^{3/4}} 
 \bigl( \|u_j(t)\|_{L^2} + \|J_{m_j}u_j(t)\|_{L^2}\bigr)
 \le
 \frac{C\eps}{t^{1/2}}
\]
for $t\ge 1$ and $j\in I_N$. 
This completes the proof of Theorem \ref{thm_sdge}.
\qed

%-----------------------------------------------
\subsection{Proof of Theorems \ref{thm_decay1} and \ref{thm_decay2}}  
\label{proof_decay}
%-----------------------------------------------

The proof of Theorems \ref{thm_decay1} and  \ref{thm_decay2} heavily 
relies on the following lemma due to \cite{KMS}. 
Note that special cases of this lemma have been used previously in 
\cite{KLS} and \cite{KimSu} less explicitly.

\begin{lem}[\cite{KMS}]  \label{lem_mats}%----------------------
Let $C_0>0$, $C_1\ge 0$, $p>1$ and $q>1$. Suppose that 
$\Psi(t)$ satisfies 
\[
 \frac{d \Psi}{dt}(t) \le \frac{-C_0}{t} |\Psi(t)|^p +\frac{C_1}{t^q}
\]
for $t\ge 2$. Then we have 
\[
 \Psi(t) \le \frac{C_2}{(\log t)^{p^*-1}}
\]
for $t\ge 2$, where $p^*$ is the H\"older conjugate of $p$ 
{\rm(}i.e., $1/p+1/p^*=1${\rm)}, and 
\[
 C_2
 = 
 \left(\frac{p^*}{C_0 p} \right)^{p^*-1}
 +
 (\log 2)^{p^*-1}\Psi(2) 
 + 
 \frac{C_1}{\log 2}
 \int_2^{\infty} \frac{(\log \tau)^{p^*}}{\tau^{q}}d\tau.
\]
\end{lem}%----------------------------------------------
\vspace{3mm}

With $\xi \in \R$ fixed, we set 
$\Psi(t)=\jb{\alpha(t,\xi), A\alpha(t,\xi)}_{\C^N}$, 
where $A$ is the positive Hermitian matrix appearing in the condition (b$_1$). 
Then we deduce from \eqref{profile} that $\Psi$ satisfies
\[
  \frac{d \Psi}{dt}(t) 
 \le 
 \frac{-2C_{*}}{t} |\alpha(t)|^4 
 +
 C|\rho(t,\xi)| |\alpha(t,\xi)|
 \le 
 \frac{-2C_{*}/\kappa_*^2}{t} |\Psi(t)|^2 
 +
 \frac{C\eps^3}{\jb{\xi}^4 t^{5/4-\gamma}}
\]
for $t\ge 2$, where $C_*$ is the positive constant appearing in 
the condition (b$_1$) and $\kappa_*$ is the smallest eigenvalue of $A$. 
We also have $\Psi(2)\le C|\alpha(2,\xi)|^2\le C\eps^2 \jb{\xi}^{-4}$. 
So we can apply Lemma \ref{lem_mats} with $p=2$, $q=5/4-\gamma$ to obtain  
\[
 |\alpha(t,\xi)|^2
 \le 
 C \Psi(t)
 \le 
 \frac{1}{(\log t)^{2-1}}\left(  
  \frac{\kappa_*^2}{2C_*} + \frac{C\eps^2}{\jb{\xi}^4}
 \right)
 \le 
 \frac{C}{\log t}.
\]
From Lemma \ref{lemma_asympt} it follows that 
\begin{align*}
 |u_j(t,x)| 
 &\le  
 \frac{C}{t^{1/2}} \sup_{\xi \in \R}|\alpha_j(t,\xi)| 
 +
 \frac{C}{t^{3/4}} \bigl( \|u_j(t)\|_{L^2} + \|J_{m_j}u_j(t)\|_{L^2}\bigr)\\
 &\le
 \frac{C}{(t \log t)^{1/2}} +\frac{C\eps}{t^{3/4-\gamma /3}} \\
 &\le
 \frac{C}{(t \log t)^{1/2}},
\end{align*}
for $t\ge 2$, $x \in \R$ and $j\in I_N$. 
On the other hand, we already know that  
$|u(t,x)| \le C \eps(1+t)^{-1/2}$ for $t\geq 0$. Hence we arrive at 
\[
 (1+t)(1+\eps^2 \log (t+2))|u(t,x)|^2 \le C \eps^2
\]
for $t\geq 0$, 
which implies the desired pointwise decay estimate. 
By the Fatou lemma we also have 
\[
 %\limsup_{t\to +\infty} \|u_j(t)\|_{L^2}^2
 %=
 \limsup_{t\to +\infty} \|\alpha_j(t)\|_{L^2}^2
 \le 
 \int_{\R} \limsup_{t\to +\infty} |\alpha_j(t,\xi)|^2 d\xi 
 =0,
\]
which leads to  decay of $\|u_j(t)\|_{L^2}$ as $t\to +\infty$, 
as stated in Theorem \ref{thm_decay1}.

Under the stronger condition (b$_2$), we have 
\[
  \frac{d \Psi}{dt}(t) 
 %\le 
 %\frac{-2C_{**}\jb{\xi}^2}{t} |\alpha(t)|^4 
 %+
 %C|\rho(t,\xi)| |\alpha(t,\xi)|
 \le 
 \frac{-2C_{**}\jb{\xi}^2/\kappa_*^2}{t} |\Psi(t)|^2 
 +
 \frac{C\eps^3}{\jb{\xi}^4 t^{5/4-\gamma}}
\]
for $t\ge 2$. Therefore Lemma \ref{lem_mats} again yields 
\[
 |\alpha(t,\xi)|^2 
 \le 
 \frac{1}{\log t}\left(  
  \frac{\kappa_*^2}{2C_{**}\jb{\xi}^2} + \frac{C\eps^2}{\jb{\xi}^4}
 \right)
 \le 
  \frac{C }{\jb{\xi}^2  \log t},
\]
whence 
\[
 \|u(t)\|_{L^2} =\|\alpha(t)\|_{L^2}
 \le 
 C \sup_{\xi \in \R} \bigl(\jb{\xi} |\alpha(t,\xi)| \bigr)
 \le 
 \frac{C}{\sqrt{\log t}}
\]
for $t\ge 2$. This yields Theorem \ref{thm_decay2}.
\qed

%-----------------------------------------------
\subsection{Proof of Theorem \ref{thm_asymp_free}}  \label{proof_tasymp_free}
%-----------------------------------------------
For given $\delta>0$, we set $\gamma=\min\{\delta, 1/5\} \in (0,1/4)$. 
Remember that we have already shown that 
\[
 |\alpha_j(t,\xi)| \le \frac{C\eps}{\jb{\xi}^2},
 \qquad
 |\rho_j(t,\xi)| 
 \le 
 \frac{C \eps^2}{\jb{\xi}^2 t^{5/4-\gamma}}
\]
for $t\ge 1$, $\xi \in \R$ and $j\in I_N$. 
These estimates allow us to define 
$\alpha^+=(\alpha_j^+)_{j\in I_N} \in L^2\cap L^{\infty}$ by 
\[
 \alpha_j^+(\xi):=\alpha_j(1,\xi) -i \int_1^{\infty} \rho_j(t',\xi) dt'.
\]
On the other hand, the condition (b$_3$) and \eqref{profile} lead to 
\[
 \alpha_j(t,\xi)=\alpha_j(1,\xi) -i \int_1^{t}\rho_j(t',\xi) dt',
\]
whence
\[
 \|\alpha_j(t) -\alpha_j^+ \|_{L^{2}\cap L^{\infty}}
 \le 
 \int_t^{\infty} \|\rho_j(t',\cdot) \|_{L^{2}\cap L^{\infty}} dt'
 \le 
 C\eps^2 t^{-1/4+{\gamma}}.
\]
Now we set 
$\varphi_j^+:=\mathcal{F}_{m_j}^{-1} \alpha_j^+$. 
Then we have
\begin{align*}
 \|u_j(t) -\mathcal{U}_{m_j}\varphi_j^+\|_{L^2}
 &=
 \|
 \mathcal{F}_{m_j}\mathcal{U}_{m_j}^{-1}u_j(t) -\mathcal{F}_{m_j}\varphi_j^+
 \|_{L^2}\\
 &=
 \|\alpha_j(t) -\alpha_j^+ \|_{L^2}\\
 &\le 
 C\eps^2 t^{-1/4+\gamma}.
\end{align*}
By Lemma \ref{lemma_asympt} and the inequality obtained above, we also have
\begin{align*}
 &\|
   u_j(t) -\mathcal{M}_{m_j} \mathcal{D}\mathcal{F}_{m_j}\varphi_j^+
 \|_{L^{\infty}}\\
 &\le
 \|
   u_j(t) -\mathcal{M}_{m_j} \mathcal{D}\mathcal{F}_{m_j}\mathcal{U}_{m_j}^{-1}
  u_j(t)
 \|_{L^{\infty}} 
 +
  \|
   \mathcal{M}_{m_j} \mathcal{D}(\alpha_j(t)-\alpha^+)
 \|_{L^{\infty}}\\
 &\le
 Ct^{-3/4} (\|u_j(t)\|_{L^2}+\|J_{m_j}u_j(t)\|_{L^2} )
 + 
 Ct^{-1/2} \|\alpha_j(t)-\alpha_j^+\|_{L^{\infty}}\\
 &\le
 C\eps t^{-3/4+\gamma/3} + C\eps^2 t^{-1/2 -1/4+\gamma}\\
 &\le 
 C\eps t^{-3/4+\delta}
\end{align*}
for $t\geq 1$. 
\qed

%-----------
\begin{rmk} \label{rmk_opt}
We put $\varphi_j=\eps' \psi_j$ with $\psi_j \not\equiv 0$ and 
$\eps' \in (0,\eps^*]$, 
where $\eps^*>0$ is chosen suitably small so that  
Theorem \ref{thm_asymp_free} is valid. Then we can check that  
the corresponding $\varphi_j^+$ satisfies 
\[
 \|\varphi_j^+\|_{L^2}
 =
 \|\alpha_j^+\|_{L^2} 
 \ge
 \eps' \|\psi_j\|_{L^2} - C^* (\eps')^3
\]
with some $C^* >0$. Therefore $\varphi_j^+$ does not identically vanish 
if $\eps'< \min\{\eps^*, \sqrt{\|\psi_j\|_{L^2} / C^*}\}$. 

\end{rmk}
%------------

%-----------------------------------------------
\appendix \section{Proof of Lemma~\ref{lemma_smoothing}} \label{appendix}
%-----------------------------------------------

In this appendix, we shall give the proof of Lemma~\ref{lemma_smoothing} 
in the similar way as Section 2 of \cite{HNP} with slight modifications.
We first state the following useful lemma without proof, which is a 
special case of Lemma 2.1 of \cite{HNP}. 

%----------------
\begin{lem} \label{lemma_commutator}
We have 
\[
 \left\| \Bigl[|\pa_x|^{1/2}, g \Bigr] f \right\|_{L^2}
 +
 \left\| \Bigl[|\pa_x|^{1/2}\mathcal{H}, g \Bigr] f \right\|_{L^2}
 \le 
 C \|g\|_{W^{1,\infty}} \|f\|_{L^2}.
\]
\end{lem}
%----------------
\vspace{5mm}

%--------------------------------------------------------------%
\noindent{\it Proof of Lemma~\ref{lemma_smoothing}.}\ \ 
As in the standard energy method, we compute
\[
 \frac{1}{2}\frac{d}{dt} \|Sv\|_{L^2}^2
 =\imagpart\jb{\mathcal{L}_m S v, Sv}_{L^2}
 =\imagpart\jb{S \mathcal{L}_m v, Sv}_{L^2}
  + 
  \imagpart\JB{[\mathcal{L}_m, S]v, Sv}_{L^2}.
\]
We also note that 
\[
 [\mathcal{L}_m,S]v=-\frac{i}{|m|}\Phi S |\pa_x|v+Q,
\]
where
\[
 Q
 =
 \frac{1}{2m}\Phi^2 Sv  
 -
 \frac{i}{2|m|}(\pa_x \Phi) S\mathcal{H}v
 +
 \mathrm{sgn}(m)
 \left(\int_{-\infty}^{x} \pa_t \Phi(t,y)dy\right) S\mathcal{H}v.
\]
Remark that $|\pa_x|=\mathcal{H}\pa_x=\pa_x \mathcal{H}$, 
$\mathcal{H}^2=-1$, and that $\mathcal{H}$ is $L^2$-bounded. 
Now we set $w_{k}^{(l)}=\pa_x^l w_k$ for $l \in \mathbb{Z}_{\ge 0}$. 
Then, since 
\begin{align*}
 \pa_t\Phi
 &=
 2\eta \sum_{l=0}^{1} \sum_{k \in I_N} 
 \imagpart \Bigl\{(i\pa_t w_{k}^{(l)})\overline{w_{k}^{(l)}}\Bigr\}\\
 &=
  2\eta \sum_{l=0}^{1} \sum_{k \in I_N} 
 \imagpart
 \left\{ \biggl(-\frac{1}{2\mu_k}\pa_x^2 w_{k}^{(l)} 
  + \pa_x^l \mathcal{L}_{\mu_k}w_k \biggr) 
 \overline{w_{k}^{(l)}} \right\}\\
 &=
  2\eta \sum_{l=0}^{1} \sum_{k \in I_N} 
 \imagpart
 \left\{ \pa_x 
 \biggl(-\frac{1}{2\mu_k}(\pa_x w_{k}^{(l)})\overline{w_{k}^{(l)}} \biggr)
 +\frac{1}{2\mu_k} \bigl| \pa_x w_{k}^{(l)} \bigr|^2 
  + (\pa_x^l \mathcal{L}_{\mu_k}w_k) \overline{w_{k}^{(l)}} 
\right\}\\
 &=
  2\eta \sum_{l=0}^{1} \sum_{k \in I_N} 
 \imagpart
 \left\{ \pa_x 
 \biggl(-\frac{1}{2\mu_k}(\pa_x w_{k}^{(l)})\overline{w_{k}^{(l)}} \biggr) 
 + (\pa_x^l \mathcal{L}_{\mu_k}w_k) \overline{w_{k}^{(l)}} 
\right\},
\end{align*}
we see that 
\begin{align*}
 \left|\int_{-\infty}^{x} \pa_t \Phi(t,y)dy\right|
 &=
  2\eta\left|
 \sum_{l=0}^{1}\sum_{k\in I_N} \imagpart
 \Bigl\{ 
  -\frac{1}{2 \mu_k}(\pa_x w_{k}^{(l)})\overline{w_{k}^{(l)}}+
  \int_{-\infty}^{x}
  \bigl(\pa_x^l \mathcal{L}_{\mu_k}w_k \bigr) \overline{w_{k}^{(l)}}
  dy
 \Bigr\}
 \right|\\
 &\le
 C\eta \Bigl(\|w\|_{W^{2,\infty}}^2
 + 
 \sum_{k\in I_N} \|\mathcal{L}_{\mu_k}w_k\|_{H^1} \|w_k\|_{H^1} \Bigr).
\end{align*}
Therefore we obtain
\begin{align}
 \frac{d}{dt} \|Sv\|_{L^2}^2
 +
 \frac{2}{|m|}\realpart\jb{\Phi S|\pa_x|v,Sv}_{L^2}
 \le 
 2\bigl| \jb{S \mathcal{L}_m v, Sv}_{L^2} \bigr|
  + 
  CB_1(t)\|Sv\|_{L^2}^2,
 \label{est_smoothing_pre1}
\end{align}
where
\begin{align*}
 B_1(t)
 &=
 e^{C \|\Phi\|_{L^1}}\Bigl(
 \|\Phi\|_{L^{\infty}}^2 + \|\pa_x \Phi\|_{L^{\infty}} 
 + 
 \eta \|w\|_{W^{2,\infty}}^2
 + 
 \eta \sum_{k\in I_N} \|\mathcal{L}_{\mu_k}w_k\|_{H^1} \|w_k\|_{H^1}
 \Bigr).
\end{align*}
Next we observe that 
\begin{align*}
 w_{k}^{(l)} S|\pa_x|v
 &=
 w_{k}^{(l)} S\pa_x \mathcal{H}v\\
 &=
 \pa_x (w_{k}^{(l)} S\mathcal{H}v) + [w_{k}^{(l)} S,\pa_x]\mathcal{H}v\\
 &=
 -|\pa_x|^{1/2}|\pa_x|^{1/2} \mathcal{H} w_{k}^{(l)} S\mathcal{H} v
 + [w_{k}^{(l)} S,\pa_x]\mathcal{H}v\\
 &=
 |\pa_x|^{1/2}\bigl( w_{k}^{(l)} S|\pa_x|^{1/2}v \bigr) 
 +
 [w_{k}^{(l)} S, \pa_x]\mathcal{H}v
  -
 |\pa_x|^{1/2} 
 \Bigl[ |\pa_x|^{1/2}\mathcal{H}, w_{k}^{(l)} S \Bigr]\mathcal{H}v,
\end{align*}
which leads to 
\begin{align*}
 \Jb{w_{k}^{(l)}S|\pa_x|v,w_{k}^{(l)}Sv}_{L^2}
 =&
 \Jb{w_{k}^{(l)}S|\pa_x|^{1/2}v, |\pa_x|^{1/2}
 \bigl(w_{k}^{(l)}S v\bigr) }_{L^2}
 +
 \Jb{[w_{k}^{(l)}S,\pa_x]\mathcal{H}v, w_{k}^{(l)}Sv}_{L^2}\\
 & -
 \JB{ \bigl[ |\pa_x|^{1/2} \mathcal{H}, w_{k}^{(l)}S \bigr] \mathcal{H}v, 
 |\pa_x|^{1/2} (w_{k}^{(l)}Sv)}_{L^2}
\\
 =&
 \Bigl\| w_{k}^{(l)} S|\pa_x|^{1/2} v \Bigr\|_{L^2}^2 
 + 
 X_{kl},
\end{align*}
where
\begin{align*}
 X_{kl}
 =&
 \JB{
   w_{k}^{(l)}S|\pa_x|^{1/2}v, \bigl[ |\pa_x|^{1/2}, w_{k}^{(l)}S \bigr] v
 }_{L^2}
 +
 \JB{ \bigl[w_{k}^{(l)}S,\pa_x \bigr] \mathcal{H}v, w_{k}^{(l)}Sv}_{L^2}\\
 &-
 \JB{ \bigl[ |\pa_x|^{1/2} \mathcal{H}, w_{k}^{(l)}S \bigr] \mathcal{H}v, 
 w_{k}^{(l)}S|\pa_x|^{1/2} v}_{L^2}
 -
 \JB{ \bigl[|\pa_x|^{1/2} \mathcal{H}, w_{k}^{(l)}S \bigr] \mathcal{H}v, 
     \bigl[ |\pa_x|^{1/2}, w_{k}^{(l)}S \bigr] v}_{L^2}.
\end{align*}
By using Lemma~\ref{lemma_commutator}, we can see that all the commutators 
appearing in $X_{kl}$ are $L^2$-bounded and their operator norms 
are dominated by 
\begin{align*}
 B_2(t)
 &=
 Ce^{C\|\Phi\|_{L^1}} 
 \bigl( 
   \|w\|_{W^{2,\infty}} + \|w\|_{W^{1,\infty}} \|\Phi\|_{L^{\infty}} 
 \bigr).
\end{align*}
Hence we obtain 
\begin{align*}
 \Bigl\| \sqrt{\Phi} S|\pa_x|^{1/2} v \Bigr\|_{L^2}^2
  &- \realpart \Jb{\Phi S|\pa_x|v,Sv}_{L^2}\\
 &=
 \sum_{l=0}^{1}\sum_{k\in I_N} \eta \realpart
 \biggl( 
  \Bigl\| w_{k}^{(l)} S|\pa_x|^{1/2} v \Bigr\|_{L^2}^2 
  - 
  \Jb{w_{k}^{(l)}S|\pa_x|v,w_{k}^{(l)}Sv}_{L^2}
 \biggr)\\
 &\le
 \sum_{l=0}^{1}\sum_{k\in I_N} \eta |X_{kl}| \\
 &\le
 C\eta B_2(t)
 \sum_{l=0}^{1}
 \sum_{k\in I_N} 
 \Bigl\| w_{k}^{(l)}S|\pa_x|^{1/2} v \Bigr\|_{L^2}  \bigl\| v \bigr\|_{L^2}
 + C\eta B_2(t)^2 \|v\|_{L^2}^2 \\
 &\le
 \frac{1}{2}\Bigl\| \sqrt{\Phi} S|\pa_x|^{1/2} v \Bigr\|_{L^2}^2
 +C\eta B_2(t)^2 \|v\|_{L^2}^2,
\end{align*}
where we have used the Young inequality in the last line. 
Therefore,
\begin{align}
 \frac{2}{|m|}\realpart \jb{\Phi S|\pa_x|v,Sv}_{L^2}
 \ge
 \frac{1}{|m|} \Bigl\| \sqrt{\Phi} S|\pa_x|^{1/2} v \Bigr\|_{L^2}^2
 -
 C\eta B_2(t)^2 \|v\|_{L^2}^2.
\label{est_smoothing_pre2}
\end{align}
From \eqref{est_smoothing_pre1} and \eqref{est_smoothing_pre2} 
it follows that 
\[
 \frac{d}{dt} \|Sv\|_{L^2}^2
 +
 \frac{1}{|m|} \Bigl\| \sqrt{\Phi} S|\pa_x|^{1/2} v \Bigr\|_{L^2}^2
 \le 
 2 \bigl| \jb{S \mathcal{L}_m v, Sv}_{L^2} \bigr|
  + 
  C \bigl( B_1(t)+\eta B_2(t)^2 \bigr) \|Sv\|_{L^2}^2.
\]
Finally, by using 
$\|\Phi\|_{L^1} \le C\eta \|w\|_{H^1}^2$,
$\|\Phi\|_{L^{\infty}}\le C \eta \|w\|_{W^{1,\infty}}^2$ 
and $\|\pa_x \Phi\|_{L^{\infty}} \le C \eta \|w\|_{W^{2,\infty}}^2$, 
we have 
\begin{align*} 
 B_1(t)+\eta B_2(t)^2
 \le&
 Ce^{C\eta \|w\|_{H^1}^2}\Bigl(
 \eta^2 \|w\|_{W^{1,\infty}}^4 + \eta\|w\|_{W^{2,\infty}}^2 
 + 
 \eta \sum_{k\in I_N} \|\mathcal{L}_{\mu_k}w_k\|_{H^1} \|w_k\|_{H^1}
 \Bigr)\\
 &+
 C\eta e^{C\eta \|w\|_{H^1}^2} 
 \bigl( 
   \|w\|_{W^{2,\infty}}^2 + C\eta^2 \|w\|_{W^{1,\infty}}^6 
 \bigr)\\
 \le&
  Ce^{C\eta \|w\|_{H^1}^2}\Bigl(
 \eta\|w\|_{W^{2,\infty}}^2 + \eta^3 \|w\|_{W^{1,\infty}}^6
 + 
 \eta \sum_{k\in I_N} \|\mathcal{L}_{\mu_k}w_k\|_{H^1} \|w_k\|_{H^1}
 \Bigr),
\end{align*}
which yields the desired conclusion. 
\qed

%--------------------------------------------------------------------
\medskip
\subsection*{Acknowledgments}
One of the authors (H.S.) would like to express his gratitude for warm 
hospitailty of Department of Mathematics, Yanbian University. Main parts of 
this work were done during his visit there. 
The authors thank Professor Soichiro Katayama for his useful 
conversations on this subject.

The work of C.~L. is supported by NNSFC under Grant No. 11461074. 
The work of H.~S. is supported by Grant-in-Aid for Scientific Research (C) 
(No.~25400161), JSPS.

%%%%%%%%%%%%%%%%%%%%%%%%%%%%%%%%%%%%%%%%%%%%%%%%%%%%%%%%%%%%%%%%%%%%%%%%%%%%%%%

%%%%%%%%%%%%%%%%%%%%%%%%%%%%%%%%%%%%%%%%%%%%%%%%%%%%%%%%%%%%%%%%%%%%%%%%%%%%%%%
\end{document}